\documentclass[a4paper]{article}

\usepackage{amscd}
\usepackage{amsmath}
\usepackage{latexsym}
\usepackage{amsfonts}
\usepackage{amssymb}
\usepackage{graphicx,cases}
\usepackage{color,cite}
\usepackage{amsthm}
\usepackage{indentfirst}
\usepackage{bm}

\makeatletter

\@addtoreset{equation}{section}
\makeatother
\allowdisplaybreaks[4]

\newtheorem{lemma}{Lemma}
\newtheorem{proposition}{Proposition}[section]
\newtheorem{theorem}{Theorem}[section]

\theoremstyle{definition}
\newtheorem{definition}{Definition}[section]
\newtheorem{remark}{Remark}[section]

\def \p {\partial}
\def \O {\varOmega}
\def \e {\tau}
\def \r {\rho}

\def \d {\cdot}
\def \g {\Gamma}
\def \n {\nabla}
\def \nh {\nabla_h}
\def \la {\Delta}

\newcommand{\abs}[1]{\left|#1\right|}
\newcommand{\xkh}[1]{\left(#1\right)}
\newcommand{\zkh}[1]{\left[#1\right]}
\newcommand{\dkh}[1]{\left\{#1\right\}}
\newcommand{\ds}[1]{\int^t_0#1ds}
\newcommand{\dt}[1]{\int^r_0#1dt}
\newcommand{\dz}[1]{\int^1_{-1}#1dz}
\newcommand{\dk}[1]{\int^z_0#1d\xi}
\newcommand{\mm}[1]{\int_{M}#1dxdy}
\newcommand{\oo}[1]{\int_{\O}#1dxdydz}
\newcommand{\norm}[1]{\left\lVert#1\right\rVert}
\newcommand{\rt}[1]{\int^r_0\int_{\O}#1dxdydzdt}

\newcommand{\ty}[1]{\int^{\infty}_0\int_{\O}#1dxdydzdt}

\begin{document}
\title{On the rigorous mathematical derivation for the viscous primitive equations with density stratification}

\author{{Xueke Pu,~{Wenli Zhou}} \\[1ex]
\normalsize{\it School of Mathematics and Information Sciences,}\\
\normalsize{\it Guangzhou University,~Guangzhou 510006,~China}\\
\normalsize{Email:~puxueke@gmail.com,~wywlzhou@163.com}\\}
\date{}

\maketitle
\begin{abstract}
In this paper,~we rigorously derive the governed equations describing the motion of stable stratified fluid,~from the mathematical point of view.~Specially,~we prove that the scaled Boussinesq equations strongly converge to the viscous primitive equations with density stratification as the aspect ration parameter goes to zero,~and the rate of convergence is of the same order as the aspect ratio parameter.~Moreover,~in order to obtain this convergence result,~we also establish the global well-posedness of strong solutions to the viscous primitive equations with density stratification.
\end{abstract}
\begin{center}
 \begin{minipage}{120mm}
\end{minipage}
\end{center}
\begin{center}
\begin{minipage}{120mm}
\end{minipage}
\end{center}

\section{Introduction}
The primitive equations are considered as the fundamental model in geophysical flows (see, e.g., \cite{wm1986,jp1987,ds1996,am2003,gk2006}).~For large-scale ocean dynamics,~an important feature is that the vertical scale of ocean is much smaller than the horizontal scale,~which means that we have to use the hydrostatic approximation to simulate the motion of ocean in the vertical direction.~Owing to this fact and the high accuracy of hydrostatic approximation,~the primitive equations of ocean dynamics can be formally derived from the Boussinesq equations (see \cite{rt1992,ct2007}).

The small aspect ratio limit from the Navier-Stokes equations to the primitive equations was studied first by Az\'{e}rad-Guill\'{e}n\cite{pa2001}~in a weak sense,~then by Li-Titi\cite{lt2019}~in a strong sense with error estimates,~and finally by Furukawa \textit{et al.}~\cite{kf2020}~in a strong sense but under relaxing the regularity on the initial condition.~Subsequently, the strong convergence of solutions of the scaled Navier-Stokes equations to the corresponding ones of the primitive equations with only horizontal viscosity was obtained by Li-Titi-Yuan\cite{yu2022}.~Furthermore, the rigorous justification of the hydrostatic approximation from the scaled Boussinesq equations to the primitive equations with full viscosity and diffusivity was obtained by Pu-Zhou\cite{pz2021}.

From a physical point of view,~fluid flow is strongly influenced by effect of stratification~(see, e.g.,~\cite{jp1987,am2003,gk2006}).~An important observation for effect of stratification is that the density of a fluid changes with depth.~In some mathematical studies,~considering the hydrodynamic equations with density stratification term can often obtain better results~(see, e.g.,\cite{ct2012,cc2014,jl2014,es2016,cc2017,es2020}).~These two facts show that density stratification term is of great significance both physically and mathematically.~Therefore,~the aim of this paper is to derive rigorously the governed equations describing the motion of stable stratified fluid,~i.e.,~the viscous primitive equations with density stratification,~from the mathematical point of view.

Let~$\O_\e=M\times(-\e,\e)$~be a $\e$-dependent domian,~where~$M$~is smooth bounded domain in~$\mathbb R^2$.~Here,~$\e=H/L$~is called the aspect ratio,~measuring the ratio of the vertical and horizontal scales of the motion,~which is usually very small.~Say,~for large-scale ocean circulation,~the ratio~$\e\sim 10^{-3}\ll 1$.

Denote by~$\nh=(\p_x,\p_y)$~the horizontal gradient operator.~Then the horizontal Laplacian operator~$\la_h$ is given by
\begin{equation*}
  \la_h=\nh \d \nh=\p_{xx}+\p_{yy}.
\end{equation*}Let us consider the anisotropic Boussinesq equations defined on~$\O_\e$
\begin{equation}\label{eq:udn}
\begin{cases}
  \p_t u+(u \d \n)u+\n\pi+\frac{g\varrho}{\r_b} \vec{k}=\mu_h \la_h u+\mu_z \p_{zz}u,\\
  \p_t \varrho+u \d \n\varrho=\kappa_h\la_h\varrho+\kappa_z\p_{zz}\varrho,\\
  \n \d u=0,
\end{cases}
\end{equation}where the three dimensional velocity field~$u=(v,w)=(v_1,v_2,w)$,~the pressure~$\pi$~and density~$\varrho$~are the unknowns.~$g$~is the gravitational acceleration and~$\r_b$~is the reference constant density.~$\vec{k}=(0,0,1)$~is unit vector pointing to the~$z$-direction.~$\mu_h$~and~$\mu_z$~represent the horizontal and vertical viscosity coefficients respectively,~while $\kappa_h$~and~$\kappa_z$~represent the horizontal and vertical heat conduction coefficients respectively.

For simplicity,~the reference constant density~$\r_b$~is set to be~$\r_b=1$.~In fact,~the anisotropic Boussinesq equations~(\ref{eq:udn})~have an elementary exact solution~$(u,\pi,\varrho)=(0,\bar{p}(z),\bar{\varrho}(z))$~satisfying the hydrostatic approximation
\begin{equation*}
  \frac{d\bar{p}(z)}{dz}+g\bar{\varrho}(z)=0.
\end{equation*}Assume that
\begin{gather*}
  p(x,y,z,t)=\pi(x,y,z,t)-\bar{p}(z),\\
  \r(x,y,z,t)=\varrho(x,y,z,t)-\bar{\varrho}(z).
\end{gather*}Then the anisotropic Boussinesq equations~(\ref{eq:udn})~become
\begin{equation}\label{eq:muh}
\begin{cases}
  \p_t v+(v \d \nh)v+w \p_z v+\nh p=\mu_h\la_h v+\mu_z\p_{zz}v,\\
  \p_t w+v \d \nh w+w\p_z w+\p_z p+g\r=\mu_h\la_h w+\mu_z\p_{zz}w,\\
  \p_t \r+v \d \nh\r+w\p_z\r+\xkh{\frac{d\bar{\varrho}(z)}{dz}}w
  =\kappa_h\la_h\r+\kappa_z\p_{zz}\r+\kappa_z\frac{d^2\bar{\varrho}(z)}{dz^2},\\
  \nh \d v+\p_z w=0.
\end{cases}
\end{equation}Let~$N=\xkh{-g\frac{d\bar{\varrho}(z)}{dz}}^{1/2}$.~If~$N>0$,~then~$N$~is called the buoyancy or Brunt-V\"ais\"al\"a frequency.~When~$\frac{d\bar{\varrho}(z)}{dz}<0$,~the density decreases with height and lighter fluid is above heavier fluid,~which is called stable stratification.

Firstly,~we transform the anisotropic Boussinesq equations~(\ref{eq:muh}),~defined on the~$\e$-dependent domain~$\O_\e$,~to the scaled Boussinesq equations defined on a fixed domain.~To this end,~we introduce the following new unknowns with subscript~$\e$
\begin{gather*}
  u_{\e}=(v_\e,w_\e),~v_\e(x,y,z,t)=v(x,y,\e z,t),\\
  w_\e(x,y,z,t)=\frac{1}{\e}w(x,y,\e z,t),~p_\e(x,y,z,t)=p(x,y,\e z,t),\\
  \r_\e(x,y,z,t)=(g\e)\r(x,y,\e z,t),~\bar{p}_\e(z)=\bar{p}(\e z),~\bar{\varrho}_\e(z)=(g\e^2)\bar{\varrho}(\e z),
\end{gather*}for any~$(x,y,z)\in\O=:M\times(-1,1)$~and for any~$t\in(0,\infty)$.~Then the last two scalings allow us to write the pressure and density non-dimensionally as
\begin{equation*}
  \bar{p}_\e(z)+p_\e(x,y,z,t)=\bar{p}(\e z)+p(x,y,\e z,t)=\pi(x,y,\e z,t)
\end{equation*}and
\begin{equation*}
  \bar{\varrho}_\e(z)+\e\r_\e(x,y,z,t)=(g\e^2)(\bar{\varrho}(\e z)+\r(x,y,\e z,t))=(g\e^2)\varrho(x,y,\e z,t),
\end{equation*}respectively.

Suppose that~$\mu_h=\kappa_h=1$~and~$\mu_z=\kappa_z=\e^2$.~Under these scalings,~the anisotropic Boussinesq equations~(\ref{eq:muh})~defined on~$\O_\e$~can be written as the following scaled Boussinesq equations
\begin{equation}\label{eq:bve}
\begin{cases}
  \p_t v_\e+(v_\e \d \nh)v_\e+w_\e \p_z v_\e+\nh p_\e=\la_h v_\e+\p_{zz}v_\e,\\
  \e(\p_t w_\e+v_\e \d \nh w_\e+w_\e \p_z w_\e)+\frac{1}{\e}(\p_z p_\e+\r_\e)=\e\la_h w_\e+\e\p_{zz}w_\e,\\
  \p_t \r_\e+v_\e \d \nh \r_\e+w_\e \p_z \r_\e+\frac{1}{\e}w_\e \frac{d\bar{\varrho}_\e}{dz}=\la_h\r_\e+\p_{zz}\r_\e+\frac{1}{\e}\frac{d^2\bar{\varrho}_\e}{dz^2},\\
  \nh \d v_\e+\p_z w_\e=0,
\end{cases}
\end{equation}defined on the fixed domain~$\O$.

When the fluid is steadily stratified,~we can assume for simplicity that~$\bar\varrho(z)=1-(1/g)N^2z$~for some positive constant~$N^2$,~where~$N$~represents the strength of the stable stratification.~This assumption leads to~$\bar\varrho_\e(z)=(g\e^2)\bar\varrho(\e z)=g\e^2-\e^3 N^2 z$,~and hence the third equation of the scaled Boussinesq equations~(\ref{eq:bve})~becomes
\begin{equation*}
  \p_t \r_\e+v_\e \d \nh\r_\e+w_\e \p_z \r_\e-\e^2 N^2 w_\e=\la_h \r_\e+\p_{zz}\r_\e.
\end{equation*}Set~$\e^2 \d N^2=1$,~i.e.,~$N \sim 1/\e$,~which means that the stratification effect is very strong.~In such a case,~the scaled Boussinesq equations~(\ref{eq:bve})~can be rewritten as
\begin{equation}\label{eq:ptv}
\begin{cases}
  \p_t v_\e-\la v_\e+(v_\e \d \nh)v_\e+w_\e \p_z v_\e+\nh p_\e=0,\\
  \e^2\xkh{\p_tw_\e-\la w_\e+v_\e \d \nh w_\e+w_\e \p_z w_\e}+\p_zp_\e+\r_\e=0,\\
  \p_t \r_\e-\la \r_\e+v_\e \d \nh \r_\e+w_\e \p_z \r_\e-w_\e=0,\\
  \nh \d v_\e+\p_z w_\e=0.
\end{cases}
\end{equation}

Next,~we supply the scaled Boussinesq equations~(\ref{eq:ptv})~with the following boundary and initial conditions
\begin{gather}
  v_\e,w_\e,p_\e~\textnormal{and}~\r_\e~\textnormal{are periodic in}~x,y,z, \label{ga:are}\\
  (v_\e,w_\e,\r_\e)|_{t=0}=(v_0,w_0,\r_0), \label{ga:vew}
\end{gather}where~$(v_0,w_0,\r_0)$~is given.~Moreover,~we also equip the system~(\ref{eq:ptv})~with the following symmetry condition
\begin{equation}\label{eq:eve}
  v_\e,w_\e,p_\e~\textnormal{and}~\r_\e~\textnormal{are even,~odd,~even and odd with respect to}~z,~\textnormal{respectively}.
\end{equation}Noting that the above symmetry condition is preserved by the scaled Boussinesq equations~(\ref{eq:ptv}),~i.e.,~it holds provided that the initial data satisfies this symmetry condition.~Due to this fact,~throughout this paper,~we always suppose that the initial data satisfies
\begin{equation}\label{eq:xyz}
  v_0,w_0,\textnormal{and}~\r_0~\textnormal{are periodic in}~x,y,z,~\textnormal{and are even,~odd,~and odd in}~z,~\textnormal{respectively}.
\end{equation}

In this paper,~we will not distinguish in notation between spaces of scalar and vector-valued functions.~Namely,~we will use the same notation to denote both a space itself and its finite product spaces.~For convenience,~we denote by notation~$\norm{\d}_p$~and~$\norm{\d}_{p,M}$~the~$L^p(\O)$~norm and~$L^p(M)$~norm,~respectively.~Moreover,~since the scaled Boussinesq equations~(\ref{eq:ptv})~satisfy the symmetry condition~(\ref{eq:eve}),~it follows from the divergence-free condition that~$w_0$~is uniquely determined as
\begin{equation}\label{eq:yxi}
  w_0(x,y,z)=-\dk{\nh \d v_0(x,y,\xi)},
\end{equation}for any~$(x,y) \in M$~and~$z \in (-1,1)$.~Hence only the initial condition of $(v_\e,\r_\e)$ is given throughout the paper.

For the proof of the existence of weak solutions to the scaled Boussinesq equations~(\ref{eq:ptv}),~we refer to the work of Lions-Temam-Wang\cite[Part IV]{rt1992}.~Specifically,~for any initial data~$(u_0,\r_0)=(v_0,w_0,\r_0) \in L^2(\O)$,~with~$\n \d u_0=0$,~we can prove that there exists a global weak solution~$(v_\e,w_\e,\r_\e)$~of the scaled Boussinesq equations~(\ref{eq:ptv}),~subject to boundary and initial conditions~(\ref{ga:are})-(\ref{ga:vew})~and symmetry condition~(\ref{eq:eve}).~Moreover,~by the similar argument as Lions-Temam-Wang~\cite[Part IV]{rt1992},~we can also show that it has a unique local strong solution~$(v_\e,w_\e,\r_\e)$~for initial data~$(u_0,\r_0)=(v_0,w_0,\r_0) \in H^1(\O)$,~with~$\n \d u_0=0$.~The weak solutions of the scaled Boussinesq equations~(\ref{eq:ptv})~are defined as follows.

\begin{definition}
Given~$(u_0,\r_0)=(v_0,w_0,\r_0) \in L^2(\O)$,~with~$\n \d u_0=0$.~We say that a space periodic function~$(v_\e,w_\e,\r_\e)$~is a weak solution of the system~(\ref{eq:ptv}),~subject to boundary and initial conditions~(\ref{ga:are})-(\ref{ga:vew})~and symmetry condition~(\ref{eq:eve}),~if\\
(i)~$(v_\e,w_\e,\r_\e) \in C([0,\infty);L^2(\O)) \cap L^2_{loc}([0,\infty);H^1(\O))$~and\\
(ii)~$(v_\e,w_\e,\r_\e)$~satisfies the following integral equality
\begin{flalign*}
  &\ty{\bigg\{(-v_\e \d \p_t \varphi_h-\e^2 w_\e \p_t\varphi_3-\r_\e \p_t \psi+\r_\e \varphi_3-w_\e \psi)\\
  &\quad+\zkh{\n v_\e:\n \varphi_h+\e^2\n w_\e \d \n \varphi_3+\n \r_\e \d \n \psi}\\
  &\quad+\zkh{(u_\e \d \n)v_\e \d \varphi_h+\e^2(u_\e \d \n w_\e)\varphi_3+(u_\e \d \n \r_\e)\psi}\bigg\}}\\
  &=\oo{\xkh{v_0 \d \varphi_h(0)+\e^2w_0\varphi_3(0)+\r_0\psi(0)}},
\end{flalign*}for any spatially periodic function~$(\varphi,\psi)=(\varphi_h,\varphi_3,\psi)$,~with~$\varphi_h=(\varphi_1,\varphi_2)$,~such that~$\n \d \varphi=0$~and~$(\varphi,\psi) \in C^{\infty}_c(\overline{\O}\times[0,\infty))$.
\end{definition}

\begin{remark}\label{re:lh}
Similar to the theory of three-dimensional Navier-Stokes equations,~e.g.,~see Temam\cite[Ch.III,~Remark 4.1]{rt1977}~and~Robinson \emph{et al.}\cite[Theorem~4.6]{jc2016},~we can prove that~$(v_\e,w_\e,\r_\e)$~satisfies the following energy inequality
\begin{flalign}
  &\frac{1}{2}\xkh{\norm{v_\e(t)}^2_2+\e^2\norm{w_\e(t)}^2_2+\norm{\r_\e(t)}^2_2}\nonumber\\
  &\qquad\quad+\ds{\xkh{\norm{\n v_\e}^2_2+\e^2\norm{\n w_\e}^2_2+\norm{\n \r_\e}^2_2}}\nonumber\\
  &\qquad\leq\frac{1}{2}\xkh{\norm{v_0}^2_2+\e^2\norm{w_0}^2_2+\norm{\r_0}^2_2}.\label{fl:vwr}
\end{flalign}for~a.e.~$t \in [0,\infty)$,~as long as the weak solution~$(v_\e,w_\e,\r_\e)$~is obtained by Galerkin method.
\end{remark}

In consequence,~this paper is to study the the small aspect ratio limit for the system~(\ref{eq:ptv}).~In other words,~when the aspect ratio~$\e \rightarrow 0$,~we are going to prove that the scaled Boussinesq equations~(\ref{eq:ptv})~converge to the following viscous primitive equations with density stratification
\begin{equation}\label{eq:ptl}
\begin{cases}
  \p_t v-\la v+(v \d \nh)v+w \p_z v+\nh p=0,\\
  \p_z p+\r=0,\\
  \p_t \r-\la \r+v \d \nh \r+w \p_z \r-w=0,\\
  \nh \d v+\p_z w=0,
\end{cases}
\end{equation}in a suitable sense,~where the density stratification term~$w$~in the third equation of system~(\ref{eq:ptl})~provides additional dissipation for this system.~Moreover,~the resulting system~(\ref{eq:ptl})~satisfies the same boundary and initial conditions~(\ref{ga:are})-(\ref{ga:vew})~and symmetry condition~(\ref{eq:eve})~as the system~(\ref{eq:ptv}).

Next we want to recall some results concerning the primitive equations.~The global existence of weak solutions of the full primitive equations was first given by Lions-Temam-Wang\cite{rt1992,jl1992,sw1995},~but the question of uniqueness to this mathematical model is still unknown except for some special cases\cite{db2003,tt2010,ik2014,jl2017,ju2017}.~Furthemore,~the existence and uniqueness of strong solutions of this mathematical model in different setting are due to Cao-Titi\cite{ct2007},~Kobelkov\cite{gm2006},~Kukavica-Ziane\cite{ik2007,mz2007},~Hieber-Kashiwabara\cite{mh2016},~Hieber~\textit{et al.}\cite{ah2016},~as well as Giga~\textit{et al.}\cite{ym2020}. Subsequently,~the study of the global strong solutions to the primitive equations is naturally carried out in the cases of partial dissipation.~More details on these cases can be found in the work of Cao-Titi\cite{ct2012},~Fang-Han\cite{dy2020},~Li-Yuan\cite{li2022},~and Cao-Li-Titi\cite{cc2014,jl2014,es2016,cc2017,es2020}.~However,~the inviscid primitive equations with or without rotation is known to be ill-posed in Sobolev spaces,~and its smooth solutions may develop singularity in finite time,~see Renardy\cite{mr2009},~Han-Kwan and Nguyen\cite{dh2016},~Ibrahim-Lin-Titi\cite{si2021},~Wong\cite{tk2015},~and Cao~\textit{et al.}\cite{cc2015}.

The rest of this paper is organized as follows.~Our main results are stated in Section 2.~In Section 3,~we establish the global well-posedness of strong solutions to the viscous primitive equations with density stratification~(\ref{eq:ptl}).~The proofs of Theorem~\ref{th:I1t}~and~\ref{th:I12}~are presented in Section 4 and Section 5,~respectively.~Some auxiliary lemmas frequently used in the proof are collected in Appendix.

\section{Main results}
Now we are to state the main results of this paper.~In order to obtain the following strong convergence results,~i.e.,~Theorem~\ref{th:I1t}~and~\ref{th:I12},~we firstly establish the global well-posedness of strong solutions to the viscous primitive equations with density stratification~(\ref{eq:ptl}).

\begin{theorem}\label{th:vrt}
Suppose that a periodic function pair~$(v_0,\r_0) \in H^1(\O)$,~with
\begin{equation*}
  \dz{\nh \d v_0(x,y,z)}=0,~\oo{v_0(x,y,z)}=0,~\textnormal{and}~\oo{\r_0(x,y,z)}=0.
\end{equation*}Then,~for any~$T>0$,~there exists a unique strong solution~$(v,\r)$~depending continuously on the initial data to the system~(\ref{eq:ptl})~on the time interval~$[0,T]$,~subject to boundary and initial conditions~(\ref{ga:are})-(\ref{ga:vew})~and symmetry condition~(\ref{eq:eve}),~such that~$(v,\r) \in C([0,T];H^1(\O)) \cap L^2([0,T];H^2(\O))$~and~$(\p_t v,\p_t \r) \in L^2([0,T];L^2(\O))$.
\end{theorem}

The existence of weak solutions to the scaled Boussinesq equations~(\ref{eq:ptv})~basically follows the proof in Lions-Temam-Wang\cite[Part IV]{rt1992}.~Specifically,~for any initial data~$(u_0,\r_0)=(v_0,w_0,\r_0) \in L^2(\O)$,~with~$\n \d u_0=0$,~we can prove that there exists a global weak solution~$(v_\e,w_\e,\r_\e)$~of the scaled Boussinesq equations~(\ref{eq:ptv}),~subject to boundary and initial conditions~(\ref{ga:are})-(\ref{ga:vew})~and symmetry condition~(\ref{eq:eve}).~Assume that initial data~$(v_0,\r_0) \in H^1(\O)$.~Using this assumption condition,~it deduces from~(\ref{eq:yxi})~that~$(v_0,w_0,\r_0) \in L^2(\O)$,~which implies that the system~(\ref{eq:ptv})~has a global weak solution~$(v_\e,w_\e,\r_\e)$.~For this case,~we have the following strong convergence theorem.

\begin{theorem}\label{th:I1t}
Given a periodic function pair~$(v_0,\r_0) \in H^1(\O)$~such that
\begin{equation*}
  \dz{\nh \d v_0(x,y,z)}=0,~\oo{v_0(x,y,z)}=0,~\textnormal{and}~\oo{\r_0(x,y,z)}=0.
\end{equation*}Suppose that~$(v_\e,w_\e,\r_\e)$~is a global weak solution of the system~(\ref{eq:ptv}),~satisfying the energy inequality (\ref{fl:vwr}),~and that~$(v,\r)$~is the unique global strong solution of the system~(\ref{eq:ptl}),~with the same boundary and initial conditions~(\ref{ga:are})-(\ref{ga:vew})~and symmetry condition~(\ref{eq:eve}).~Let
\begin{equation*}
  (V_\e,W_\e,\g_\e)=(v_\e-v,w_\e-w,\r_\e-\r).
\end{equation*}Then,~for any~$T>0$,~the following estimate holds
\begin{equation*}
  \sup_{0 \leq t \leq T}\xkh{\norm{(V_\e,\e W_\e,\g_\e)}^2_2}(t)
  +\int^T_0{\norm{\n(V_\e,\e W_\e,\g_\e)}^2_2}dt \leq \e^2 \widetilde{\mathcal{K}_1}(T),
\end{equation*}where~$\widetilde{\mathcal{K}_1}(t)$~is a nonnegative continuously increasing function that does not depend on~$\e$.~As a result,~we have the following strong convergences
\begin{gather*}
  (v_\e,\e w_\e,\r_\e) \rightarrow (v,0,\r),~in~L^{\infty}\xkh{[0,T];L^2(\O)},\\
  (\n v_\e,\e \n w_\e,\n \r_\e,w_\e) \rightarrow (\n v,0,\n \r,w),~in~L^2\xkh{[0,T];L^2(\O)},
\end{gather*}and the rate of convergence is of the order~$O(\e)$.
\end{theorem}

Next,~we suppose that the initial data~$(v_0,\r_0)$~belongs to~$H^2(\O)$.~Then from~(\ref{eq:yxi})~it follows that~$(v_0,w_0,\r_0)$~belongs to~$H^1(\O)$.~By the similar argument as Lions-Temam-Wang~\cite[Part IV]{rt1992},~there exists a unique local strong solution~$(v_\e,w_\e,\r_\e)$~to the system~(\ref{eq:ptv}),~subject to the boundary and initial conditions~(\ref{ga:are})-(\ref{ga:vew})~and symmetry condition~(\ref{eq:eve}).~So we denote by~$T^*_\e$~the maximal existence time of the local strong solution~$(v_\e,w_\e,\r_\e)$~to the system~(\ref{eq:ptv}).~In this case,~we also have the following strong convergence theorem.

\begin{theorem}\label{th:I12}
Given a periodic function~$(v_0,\r_0) \in H^2(\O)$~such that
\begin{equation*}
  \dz{\nh \d v_0(x,y,z)}=0,~\oo{v_0(x,y,z)}=0,~\textnormal{and}~\oo{\r_0(x,y,z)}=0.
\end{equation*}Suppose that~$(v_\e,w_\e,\r_\e)$~is the unique local strong solution of the system~(\ref{eq:ptv}),~and that~$(v,\r)$~is the unique global strong solution of the system~(\ref{eq:ptl}),~with the same boundary and initial conditions~(\ref{ga:are})-(\ref{ga:vew})~and symmetry condition~(\ref{eq:eve}).~Let
\begin{equation*}
  (V_\e,W_\e,\g_\e)=(v_\e-v,w_\e-w,\r_\e-\r).
\end{equation*}Then,~for any~$T>0$,~there is a small positive constant~$\e(T)=\frac{3\beta_0}{4\sqrt{\widetilde{\mathcal{K}_2}(T)}}$~such that the system~(\ref{eq:ptv})~exists a unique strong solution~$(v_\e,w_\e,\r_\e)$~on the time interval~$[0,T]$,~and that the system~(\ref{fl:Ve0})-(\ref{fl:zWe})~(see Section 5,~below)~has the following estimate
\begin{equation*}
  \sup_{0 \leq t \leq T}\xkh{\norm{(V_\e,\e W_\e,\g_\e)}^2_{H^1}}(t)
  +\int^{T}_0{\norm{\n(V_\e,\e W_\e,\g_\e)}^2_{H^1}}dt \leq \e^2\widetilde{\mathcal{K}_3}(T),
\end{equation*}provided that $\e \in (0,\e(T))$,~where~$\widetilde{\mathcal{K}_3}(t)$~is a nonnegative continuously increasing function that does not depend on~$\e$.~As a result,~we have the following strong convergences
\begin{gather*}
  (v_\e,\e w_\e,\r_\e) \rightarrow (v,0,\r),~in~L^{\infty}\xkh{[0,T];H^1(\O)},\\
  (\n v_\e,\e \n w_\e,\n \r_\e,w_\e) \rightarrow (\n v,0,\n \r,w),~in~L^2\xkh{[0,T];H^1(\O)},\\
  w_\e \rightarrow w,~in~L^{\infty}\xkh{[0,T];L^2(\O)},
\end{gather*}and the rate of convergence is of the order~$O(\e)$.
\end{theorem}

\begin{remark}
It should be pointed out that the case of~$\bar{\varrho}(z)=Constant$~has been studied by the authors~(see\cite{pz2021}).~Compared with\cite{pz2021},~the resulting limit system here contains the density stratification term~$w$.~In order to establish the~$H^1$~\textit{priori}~estimate on the system~(\ref{eq:ptl}),~we have to deal with equation~(\ref{fl:dkr})~and~(\ref{fl:lar}),~simultaneously.~In this way,~the global well-posedness of strong solutions to the viscous primitive equations with density stratification~(\ref{eq:ptl})~is obtained,~and the~$H^1$~\textit{priori}~estimate on strong solutions will be used in the proof of~Theorem~\ref{th:I1t}.~Moreover,~Theorem~\ref{th:I12}~is proved by establishing the second order energy estimate on strong solutions of the system~(\ref{eq:ptl}).
\end{remark}

\section{Global well-posedness of the primitive equations}
In this section,~we establish the global well-posedness of strong solutions to the viscous primitive equations with density stratification~(\ref{eq:ptl}),~subject to boundary and initial conditions~(\ref{ga:are})-(\ref{ga:vew})~and symmetry condition~(\ref{eq:eve}).

Before this,~we firstly use the symmetry condition~(\ref{eq:eve})~to reformulate the system~(\ref{eq:ptl}).~This symmetry condition indicates~$w|_{z=0}=0$.~Integrating the last equation to the system~(\ref{eq:ptl})~with respect to~$z$~yields
\begin{equation*}
  w(x,y,z,t)=-\dk{\nh \d v(x,y,\xi,t)}.
\end{equation*}~We integrate the second equation of system~(\ref{eq:ptl})~with respect to~$z$~to obtain
\begin{equation*}
  p(x,y,z,t)=p_\gamma(x,y,t)-\dk{\r(x,y,\xi,t)},
\end{equation*}in which~$p_\gamma(x,y,t)$~represents unknown surface pressure as~$z=0$.~Based on the above relations,~we can recast the system~(\ref{eq:ptl})~as
\begin{flalign}
  &\p_t v-\la v+(v \d \nh)v-\xkh{\dk{\nh \d v(x,y,\xi,t)}} \p_z v+\nh p_\gamma(x,y,t)\nonumber\\
  &-\dk{\nh \r(x,y,\xi,t)}=0, \label{fl:dkr}\\
  &\p_t \r-\la \r+v \d \nh \r-\xkh{\dk{\nh \d v(x,y,\xi,t)}} \p_z \r\nonumber\\
  &+\dk{\nh \d v(x,y,\xi,t)}=0, \label{fl:lar}
\end{flalign}satisfying boundary and initial conditions
\begin{gather*}
  v~\textnormal{and}~\r~\textnormal{are periodic in}~x,y,z, \\
  (v,\r)|_{t=0}=(v_0,\r_0),
\end{gather*}and symmetry condition
\begin{equation*}
  v~\textnormal{and}~\r~\textnormal{are even and odd with respect to}~z,~\textnormal{respectively}.
\end{equation*}

\subsection{$L^2$~estimates on~$v$~and~$\r$.}
Taking the~$L^2(\O)$~inner product of the equation~(\ref{fl:dkr})~and~(\ref{fl:lar})~with~$v$~and~$\r$~respectively,~and integrating by parts,~we reach
\begin{flalign*}
  \frac{1}{2}\frac{d}{dt}&\xkh{\norm{v}^2_2+\norm{\r}^2_2}+\norm{\n v}^2_2+\norm{\n \r}^2_2\\
  &=\oo{\xkh{\dk{\nh \r(x,y,\xi,t)}} \d v}\\
  &\quad-\oo{\xkh{\dk{\nh \d v(x,y,\xi,t)}} \r}=0,
\end{flalign*}where we have used the following facts that
\begin{gather*}
  \oo{\zkh{\xkh{v \d \nh}v-\xkh{\dk{\nh \d v(x,y,\xi,t)}} \p_z v} \d v}=0,\\
  \oo{\zkh{v \d \nh \r-\xkh{\dk{\nh \d v(x,y,\xi,t)}} \p_z \r}\r}=0,
\end{gather*}and
\begin{equation*}
  \oo{\nh p_\gamma(x,y,t) \d v}=0.
\end{equation*}Integrating the differential equation above in time between~$0$~to~$t$,~we have
\begin{equation}\label{eq:e1t}
  \xkh{\norm{v}^2_2+\norm{\r}^2_2}(t)+\ds{\xkh{\norm{\n v}^2_2+\norm{\n \r}^2_2}}\leq \eta_1,
\end{equation}where
\begin{equation*}
  \eta_1=C\xkh{\norm{v_0}^2_{H^1}+\norm{\r_0}^2_{H^1}}.
\end{equation*}

\subsection{$L^4$~estimates on~$v$~and~$\r$}
Multiplying the equation~(\ref{fl:dkr})~and~(\ref{fl:lar})~by~$|v|^2 v$~and~$|\r|^2 \r$~respectively,~and integrating over~$\O$,~then it follows from integration by parts that
\begin{flalign}
  \frac{1}{4}\frac{d}{dt}&\norm{v}^4_4+\oo{|v|^2 \xkh{|\n v|^2+2\abs{\n|v|}^2}}\nonumber\\
  &\quad+\frac{1}{4}\frac{d}{dt}\norm{\r}^4_4+\oo{|\r|^2 \xkh{|\n \r|^2+2\abs{\n|\r|}^2}}\nonumber\\
  &=\oo{\xkh{\dk{v(x,y,\xi,t)}} \d \zkh{\nh \xkh{|\r|^2 \r}}}\nonumber\\
  &\quad-\oo{\xkh{\dk{\r(x,y,\xi,t)}} (\nh \d |v|^2 v)}\nonumber\\
  &\quad-\oo{\nh p_\gamma(x,y,t) \d |v|^2 v}\nonumber\\
  &=:D_1+D_2+D_3,\label{fl:v44}
\end{flalign}note that we have used the following facts that
\begin{gather*}
  \oo{\zkh{\xkh{v \d \nh}v-\xkh{\dk{\nh \d v(x,y,\xi,t)}}\p_z v} \d |v|^2 v}=0,\\
  \oo{\zkh{v \d \nh\r-\xkh{\dk{\nh \d v(x,y,\xi,t)}}\p_z \r} |\r|^2 \r}=0.
\end{gather*}We now estimate the first integral term~$D_1$~on the right-hand side of~(\ref{fl:v44}).~Using the H\"{o}lder inequality yields
\begin{flalign*}
  &D_1:=\oo{\xkh{\dk{v(x,y,\xi,t)}} \d \zkh{\nh \xkh{|\r|^2 \r}}}\\
  &\qquad\leq C\mm{\xkh{\dz{|v|}} \xkh{\dz{|\r|^2 |\nh \r|}}}\\
  &\qquad\leq C\mm{\xkh{\dz{|v|}} \xkh{\dz{|\r|^2}}^{1/2} \xkh{\dz{|\r|^2 |\nh \r|^2}}^{1/2}}\\
  &\qquad\leq C\xkh{\oo{|v|^4}}^{1/4} \xkh{\oo{|\r|^4}}^{1/4} \xkh{\oo{|\r|^2 |\nh \r|^2}}^{1/2}\\
  &\qquad\leq C\norm{v}_4 \norm{\r}_4 \xkh{\oo{|\r|^2 |\n \r|^2}}^{1/2}.
\end{flalign*}Due to the Young inequality,~we have
\begin{flalign}
  D_1&\leq C\norm{v}^2_4 \norm{\r}^2_4+\frac{3}{8}\oo{|\r|^2 |\n \r|^2}\nonumber\\
  &\leq C\xkh{\norm{v}^4_4+\norm{\r}^4_4}+\frac{3}{8}\oo{|\r|^2 |\n \r|^2}.\label{fl:r2n}
\end{flalign}A similar argument as integral term~$D_1$ gives
\begin{flalign}
  D_2:&=\oo{\xkh{-\dk{\r(x,y,\xi,t)}} (\nh \d |v|^2 v)}\nonumber\\
  &\leq C\xkh{\norm{v}^4_4+\norm{\r}^4_4}+\frac{3}{8}\oo{|v|^2 |\n v|^2}.\label{fl:v2n}
\end{flalign}For the last integral term $D_3$~on the right-hand side of~(\ref{fl:v44}),~we use the Lemma~\ref{le:phi}~and Poincar\'{e} inequality to reach
\begin{flalign}
  D_3:&=-\oo{\nh p_\gamma(x,y,t) \d |v|^2 v}\nonumber\\
  &\leq \mm{|\nh p_\gamma(x,y,t)| \xkh{\dz{|v||v|^2}}}\nonumber\\
  &\leq C\norm{v}^{1/2}_2\norm{\n v}^{1/2}_2\xkh{\norm{v}^2_4
  +\norm{v}_4\norm{|v| \n v}^{1/2}_2}\norm{\nh p_\gamma}_{2,M}.\label{fl:sxy}
\end{flalign}Applying the operator~$\textnormal{div}_h$~to the equation~(\ref{fl:dkr})~and integrating the resulting equation with respect to~$z$~from~$-1$~to~$1$,~we can see that~$p_\gamma(x,y,t)$~satisfies the following system
\begin{equation*}
\begin{cases}
  -\la_h p_\gamma=\frac{1}{2}\dz{\nh \d \zkh{(\nh \d (v \otimes v))-\dk{\nh \r}}},\\
  \mm{p_\gamma(x,y,t)}=0,~p_\gamma~\textnormal{is periodic in}~x,y,
\end{cases}
\end{equation*}where the condition~$\mm{p_\gamma(x,y,t)}=0$~is imposed to guarantee the uniqueness of $p_\gamma(x,y,t)$.~By virtue of the elliptic estimates and Poincar\'{e} inequality,~we obtain
\begin{flalign}
  \norm{\nh p_\gamma}_{2,M} &\leq C\norm{\dz{\nh \d \xkh{(\nh \d (v \otimes v))-\dk{\nh \r}}}}_{2,M}\nonumber\\
  &\leq C\xkh{\norm{\nh \d (v \otimes v)}_2+\norm{\nh\r}_2}\nonumber\\
  &\leq C\xkh{\norm{|v| \n v}_2+\norm{\n\r}_2}.\label{fl:s2M}
\end{flalign}Substituting~(\ref{fl:s2M})~into~(\ref{fl:sxy})~and then using the Young inequality yield
\begin{flalign}
  D_3&\leq C\norm{v}^{1/2}_2\norm{\n v}^{1/2}_2\xkh{\norm{v}^2_4+\norm{v}_4\norm{|v| \n v}^{1/2}_2}\xkh{\norm{\n\r}_2+\norm{|v| \n v}_2}\nonumber\\
  &\leq C\norm{v}^{1/2}_2\norm{\n v}^{1/2}_2 \xkh{\norm{\n\r}_2\norm{v}^2_4+\norm{\n\r}_2\norm{v}_4\norm{|v| \n v}^{1/2}_2}\nonumber\\
  &\quad+C\norm{v}^{1/2}_2\norm{\n v}^{1/2}_2 \xkh{\norm{v}^2_4\norm{|v| \n v}_2+\norm{v}_4\norm{|v| \n v}^{3/2}_2}\nonumber\\
  &\leq C\xkh{\norm{v}_2\norm{\n v}_2+\norm{v}^2_2\norm{\n v}^2_2+\norm{\n \r}^2_2}\xkh{\norm{v}^4_4+\norm{\r}^4_4}\nonumber\\
  &\quad+\frac{1}{4}\xkh{\norm{v}_2\norm{\n v}_2+\norm{\n \r}^2_2}+\frac{1}{4}\norm{|v| \n v}^2_2\label{fl:vnv}
\end{flalign}

Adding~(\ref{fl:r2n}),~(\ref{fl:v2n})~and~(\ref{fl:vnv})~leads to
\begin{flalign*}
  \frac{d}{dt}&\xkh{\norm{v}^4_4+\norm{\r}^4_4}+\oo{\xkh{|v|^2 |\n v|^2+|\r|^2 |\n \r|^2}}\\
  &\leq C\xkh{\norm{v}_2\norm{\n v}_2+\norm{v}^2_2\norm{\n v}^2_2+\norm{\n \r}^2_2+1}\xkh{\norm{v}^4_4+\norm{\r}^4_4}\\
  &\quad+\xkh{\norm{v}_2\norm{\n v}_2+\norm{\n \r}^2_2}.
\end{flalign*}Owing to the Gronwall inequality,~it deduce from~(\ref{eq:e1t})~that
\begin{flalign}
  &\xkh{\norm{v}^4_4+\norm{\r}^4_4}(t)+\ds{\oo{\xkh{|v|^2 |\n v|^2+|\r|^2 |\n \r|^2}}}\nonumber\\
  &\qquad\leq \exp\dkh{C\ds{\xkh{\norm{v}_2\norm{\n v}_2+\norm{v}^2_2\norm{\n v}^2_2+\norm{\n \r}^2_2+1}}}\nonumber\\
  &\qquad\quad\times\zkh{\norm{v_0}^4_4+\norm{\r_0}^4_4+\ds{\xkh{\norm{v}_2\norm{\n v}_2+\norm{\n \r}^2_2}}}\leq \eta_2(t),\label{fl:e2t}
\end{flalign}where
\begin{equation*}
  \eta_2(t)=(t+2)e^{C(t+2)\xkh{\eta^2_1+\eta_1+1}}\zkh{\norm{v_0}^4_{H^1}+\norm{\r_0}^4_{H^1}+\eta_1}.
\end{equation*}

\subsection{$L^2$~estimates on~$\p_z v$~and~$\p_z \r$}
Taking the~$L^2(\O)$~inner product of the equation~(\ref{fl:dkr})~and~(\ref{fl:lar})~with~$-\p_{zz} v$~and~$-\p_{zz}\r$~respectively,~we obtain
\begin{flalign}
  \frac{1}{2}\frac{d}{dt}&\xkh{\norm{\p_z v}^2_2+\norm{\p_z \r}^2_2}+\norm{\n \p_z v}^2_2+\norm{\n \p_z \r}^2_2\nonumber\\
  &=\oo{\zkh{\xkh{\dk{\nh \d v(x,y,\xi,t)}}\p_{zz}\r-\xkh{\dk{\nh \r(x,y,\xi,t)}} \d \p_{zz}v}}\nonumber\\
  &\quad+\oo{\zkh{v \d \nh\r-\xkh{\dk{\nh \d v(x,y,\xi,t)}}\p_z \r} \p_{zz}\r}\nonumber\\
  &\quad+\oo{\zkh{\xkh{v \d \nh}v-\xkh{\dk{\nh \d v(x,y,\xi,t)}}\p_z v} \d \p_{zz} v}\nonumber\\
  &=:D_1+D_2+D_3.\label{fl:zzv}
\end{flalign}For the first integral term~$D_1$~on the right-hand side of~(\ref{fl:zzv}),~we use the integration by parts,~H\"{o}lder inequality and Young inequality to reach
\begin{flalign*}
  D_1:&=\oo{\zkh{\xkh{\dk{\nh \d v}}\p_{zz}\r-\xkh{\dk{\nh \r}} \d \p_{zz}v}}\\
  &=\oo{(v \d \nh \p_z \r-\r\nh \d \p_z v)}\\
  &\leq \norm{v}_2\norm{\nh \p_z \r}_2+\norm{\r}_2\norm{\nh \p_z v}_2\\
  &\leq C\xkh{\norm{v}^2_2+\norm{\r}^2_2}+\frac{1}{6}\xkh{\norm{\n \p_z v}^2_2+\norm{\n \p_z \r}^2_2}.
\end{flalign*}In order to obtain the upper bound for the second integral term~$D_2$~on the right-hand side of~(\ref{fl:zzv}),~using the integration by parts,~H\"{o}lder inequality,~Lebesgue interpolation inequality,~Sobolev embedding,~as well as Poincar\'{e} inequality gives
\begin{flalign*}
  D_2:&=\oo{\zkh{v \d \nh\r-\xkh{\dk{\nh \d v(x,y,\xi,t)}}\p_z \r} \p_{zz}\r}\\
  &=\oo{\zkh{\xkh{\nh \d \p_z v}\r \p_z \r+\xkh{\p_z v \d \nh \p_z \r}\r-2\xkh{v \d \nh \p_z \r}\p_z \r}}\\
  &\leq \norm{\r}_4 \norm{\p_z \r}_4 \norm{\nh \p_z v}_2
  +\xkh{\norm{\r}_4 \norm{\p_z v}_4+\norm{v}_4 \norm{\p_z \r}_4}\norm{\nh \p_z \r}_2\\
  &\leq C\norm{\r}_4 \norm{\p_z \r}^{1/4}_2 \norm{\n \p_z \r}^{3/4}_2 \norm{\n \p_z v}_2
  +C\norm{v}_4 \norm{\p_z \r}^{1/4}_2 \norm{\n \p_z \r}^{7/4}_2\\
  &\quad+C\norm{\r}_4 \norm{\p_z v}^{1/4}_2 \norm{\n \p_z v}^{3/4}_2 \norm{\n \p_z \r}_2.
\end{flalign*}By virtue of the Young inequality,~we have
\begin{equation*}
  D_2\leq C\xkh{\norm{v}^8_4+\norm{\r}^8_4}\xkh{\norm{\p_z v}^2_2+\norm{\p_z \r}^2_2}
  +\frac{1}{6}\xkh{\norm{\n \p_z v}^2_2+\norm{\n \p_z \r}^2_2}.
\end{equation*}With the similar argument of the second integral terms~$D_2$~on the right-hand side of~(\ref{fl:zzv}),~the last integral term~$D_3$~can be estimated as
\begin{flalign*}
  D_3:&=\oo{\zkh{\xkh{v \d \nh}v-\xkh{\dk{\nh \d v(x,y,\xi,t)}}\p_z v} \d \p_{zz} v}\\
  &=\oo{\zkh{\xkh{\nh \d \p_z v}v \d \p_z v+\xkh{\p_z v \d \nh}\p_z v \d v-2\xkh{v \d \nh}\p_z v \d \p_z v}}\\
  &\leq C\norm{v}^8_4\xkh{\norm{\p_z v}^2_2+\norm{\p_z \r}^2_2}+\frac{1}{6}\norm{\n \p_z v}^2_2.
\end{flalign*}Combining the estimates for~$D_1$,~$D_2$~and~$D_3$~leads to
\begin{flalign*}
  \frac{d}{dt}&\xkh{\norm{\p_z v}^2_2+\norm{\p_z \r}^2_2}+\norm{\n \p_z v}^2_2+\norm{\n \p_z \r}^2_2\\
  &\qquad\leq C\xkh{\norm{v}^8_4+\norm{\r}^8_4}\xkh{\norm{\p_z v}^2_2+\norm{\p_z \r}^2_2}
  +C\xkh{\norm{v}^2_2+\norm{\r}^2_2}.
\end{flalign*}Using the Gronwall inequality,~it follows from~(\ref{eq:e1t})~and~(\ref{fl:e2t})~that
\begin{flalign}
  &\xkh{\norm{\p_z v}^2_2+\norm{\p_z \r}^2_2}(t)+\ds{\xkh{\norm{\n \p_z v}^2_2+\norm{\n \p_z \r}^2_2}}\nonumber\\
  &\qquad\leq \exp\dkh{C\ds{\xkh{\norm{v}^8_4+\norm{\r}^8_4}}}\nonumber\\
  &\qquad\quad\times\zkh{\norm{\p_z v_0}^2_2+\norm{\p_z \r_0}^2_2+C\ds{\xkh{\norm{v}^2_2+\norm{\r}^2_2}}} \leq \eta_3(t),\label{fl:e3t}
\end{flalign}where
\begin{equation*}
  \eta_3(t)=C(t+1)e^{Ct\eta^2_2(t)}\zkh{\norm{v_0}^2_{H^1}+\norm{\r_0}^2_{H^1}+\eta_1}.
\end{equation*}

\subsection{$L^2$~estimates on~$\n v$~and~$\n \r$}
Multiplying the equation~(\ref{fl:dkr})~and~(\ref{fl:lar})~by~$\p_t v-\la v$~and~$\p_t \r-\la \r$~respectively,~integrating over~$\O$,~and integrating by parts,~we reach
\begin{flalign}
  \frac{d}{dt}&\xkh{\norm{\n v}^2_2+\norm{\n \r}^2_2}+\norm{\p_t v}^2_2+\norm{\la v}^2_2+\norm{\p_t \r}^2_2+\norm{\la \r}^2_2\nonumber\\
  &=\oo{\zkh{\xkh{\dk{\nh \d v(x,y,\xi,t)}} \xkh{\la \r-\p_t \r}-\xkh{\dk{\nh \r(x,y,\xi,t)}} \d \xkh{\la v-\p_t v}}}\nonumber\\
  &\quad+\oo{\zkh{v \d \nh \r \xkh{\la \r-\p_t \r}+\xkh{v \d \nh}v \d \xkh{\la v-\p_t v}}}\nonumber\\
  &\quad+\oo{\xkh{\dk{\nh \d v(x,y,\xi,t)}}\zkh{\p_z \r \xkh{\p_t \r-\la \r}+\p_z v \d \xkh{\p_t v-\la v}}}\nonumber\\
  &=:D_1+D_2+D_3.\label{fl:lav}
\end{flalign}Due to the H\"{o}lder inequality and Young inequality,~the first integral term~$D_1$~on the right-hand side of~(\ref{fl:lav})~can be bounded as
\begin{flalign}
  &D_1:=\oo{\zkh{\xkh{\dk{\nh \d v(x,y,\xi,t)}} \xkh{\la \r-\p_t \r}-\xkh{\dk{\nh \r(x,y,\xi,t)}} \d \xkh{\la v-\p_t v}}}\nonumber\\
  &\leq \mm{\zkh{\xkh{\dz{|\nh v|}}\xkh{\dz{(|\p_t \r|+|\la \r|)}}+\xkh{\dz{|\nh \r|}}\xkh{\dz{(|\p_t v|+|\la v|)}}}}\nonumber\\
  &\leq C\norm{\nh v}_2\xkh{\norm{\p_t \r}_2+\norm{\la \r}_2}+C\norm{\nh \r}_2\xkh{\norm{\p_t v}_2+\norm{\la v}_2}\nonumber\\
  &\leq C\xkh{\norm{\n v}^2_2+\norm{\n \r}^2_2}+\frac{1}{6}\xkh{\norm{\p_t v}^2_2+\norm{\la v}^2_2+\norm{\p_t \r}^2_2+\norm{\la \r}^2_2}.\label{fl:hdv}
\end{flalign}In order to estimate the second integral term~$D_2$~on the right-hand side of~(\ref{fl:lav}),~we use the Lemma~\ref{le:phi},~Poincar\'{e} inequality and Young inequality to obtain
\begin{flalign}
  D_2:&=\oo{\zkh{v \d \nh \r \xkh{\la \r-\p_t \r}+\xkh{v \d \nh}v \d \xkh{\la v-\p_t v}}}\nonumber\\
  &\leq \mm{\xkh{\dz{\xkh{|v|+|\p_z v|}}} \xkh{\dz{|\nh \r|\xkh{|\p_t \r|+|\la \r|}}}}\nonumber\\
  &\quad+\mm{\xkh{\dz{\xkh{|v|+|\p_z v|}}} \xkh{\dz{|\nh v|\xkh{|\p_t v|+|\la v|}}}}\nonumber\\
  &\leq C\xkh{\norm{v}^{1/2}_2 \norm{\n v}^{1/2}_2+\norm{\p_z v}^{1/2}_2 \norm{\n \p_z v}^{1/2}_2}
  \norm{\n \r}^{1/2}_2 \norm{\la \r}^{1/2}_2 \xkh{\norm{\p_t \r}_2+\norm{\la \r}_2}\nonumber\\
  &\quad+C\xkh{\norm{v}^{1/2}_2 \norm{\n v}^{1/2}_2+\norm{\p_z v}^{1/2}_2 \norm{\n \p_z v}^{1/2}_2}
  \norm{\n v}^{1/2}_2 \norm{\la v}^{1/2}_2 \xkh{\norm{\p_t v}_2+\norm{\la v}_2}\nonumber\\
  &\leq C\xkh{\norm{v}^2_2\norm{\n v}^2_2+\norm{\p_z v}^2_2 \norm{\n \p_z v}^2_2}\xkh{\norm{\n v}^2_2+\norm{\n \r}^2_2}\nonumber\\
  &\quad+\frac{1}{6}\xkh{\norm{\p_t v}^2_2+\norm{\la v}^2_2+\norm{\p_t \r}^2_2+\norm{\la \r}^2_2}.\label{fl:vdn}
\end{flalign}Finally,~it remains to estimate the last integral term~$D_3$~on the right-hand side of~(\ref{fl:lav}).~A similar argument as~$D_2$~yields
\begin{flalign}
  D_3:&=\oo{\xkh{\dk{\nh \d v(x,y,\xi,t)}}\zkh{\p_z \r \xkh{\p_t \r-\la \r}+\p_z v \d \xkh{\p_t v-\la v}}}\nonumber\\
  &\leq C\xkh{\norm{\p_z v}^2_2 \norm{\n \p_z v}^2_2+\norm{\p_z \r}^2_2 \norm{\n \p_z \r}^2_2}\xkh{\norm{\n v}^2_2+\norm{\n \r}^2_2}\nonumber\\
  &\quad+\frac{1}{6}\xkh{\norm{\p_t v}^2_2+\norm{\la v}^2_2+\norm{\p_t \r}^2_2+\norm{\la \r}^2_2}.\label{fl:nhd}
\end{flalign}Summing~(\ref{fl:hdv}),~(\ref{fl:vdn})~and~(\ref{fl:nhd})~leads to
\begin{flalign}
  \frac{d}{dt}&\xkh{\norm{\n v}^2_2+\norm{\n \r}^2_2}+\frac{1}{2}\xkh{\norm{\p_t v}^2_2+\norm{\la v}^2_2+\norm{\p_t \r}^2_2+\norm{\la \r}^2_2}\nonumber\\
  &\leq C\xkh{\norm{v}^2_2\norm{\n v}^2_2+\norm{\p_z v}^2_2 \norm{\n \p_z v}^2_2+\norm{\p_z \r}^2_2 \norm{\n \p_z \r}^2_2+1}\xkh{\norm{\n v}^2_2+\norm{\n \r}^2_2}.
\end{flalign}Applying the Gronwall inequality to the above inequality,~it follows from~(\ref{eq:e1t})~and~(\ref{fl:e3t})~that
\begin{flalign}
  &\xkh{\norm{\n v}^2_2+\norm{\n \r}^2_2}(t)+\ds{\xkh{\norm{\p_t v}^2_2+\norm{\la v}^2_2+\norm{\p_t \r}^2_2+\norm{\la \r}^2_2}}\nonumber\\
  &\qquad\leq \exp\dkh{C\ds{\xkh{\norm{v}^2_2\norm{\n v}^2_2+\norm{\p_z v}^2_2 \norm{\n \p_z v}^2_2+\norm{\p_z \r}^2_2 \norm{\n \p_z \r}^2_2+1}}}\nonumber\\
  &\qquad\quad\times\xkh{\norm{\n v_0}^2_2+\norm{\n \r_0}^2_2} \leq \eta_4(t),\label{fl:e4t}
\end{flalign}where
\begin{equation*}
  \eta_4(t)=e^{C(t+2)\xkh{\eta^2_1+\eta^2_3(t)+1}}\xkh{\norm{v_0}^2_{H^1}+\norm{\r_0}^2_{H^1}}.
\end{equation*}

Based on the above energy estimates, we give the proof of Theorem~\ref{th:vrt}.
\begin{proof}[Proof of Theorem~\ref{th:vrt}.]
Adding~(\ref{eq:e1t})~and~(\ref{fl:e4t}),~we have
\begin{equation*}
  \sup_{0 \leq s \leq t}\xkh{\norm{v}^2_{H^1}+\norm{\r}^2_{H^1}}(s)+\ds{\xkh{\norm{\p_t v}^2_2+\norm{\n v}^2_{H^1}+\norm{\p_t \r}^2_2+\norm{\n \r}^2_{H^1}}} \leq \eta_1+\eta_4(t),
\end{equation*}where~$\eta_4(t)$~is a nonnegative continuously increasing function defined on $[0,\infty)$.~For any~$T>0$,~the following estimate holds
\begin{equation*}
  \sup_{0 \leq t \leq T}\xkh{\norm{v}^2_{H^1}+\norm{\r}^2_{H^1}}(t)
  +\int^T_0{\xkh{\norm{\p_t v}^2_2+\norm{\n v}^2_{H^1}+\norm{\p_t \r}^2_2+\norm{\n \r}^2_{H^1}}}dt \leq \eta_1+\eta_4(T).
\end{equation*}In consequence,~the strong solution~$(v,\r)$~exists globally in time,~with~$(v,\r) \in C([0,T];H^1(\O)) \cap L^2([0,T];H^2(\O))$~and~$(\p_t v,\p_t \r) \in L^2([0,T];L^2(\O))$.

Moreover,~the continuous dependence on the initial data and uniqueness of strong solutions are due to the similar argument in Cao-Titi\cite{ct2007},~and so we omit the proof here.
\end{proof}

\section{Strong convergence for~$H^1$~initial data}
In this section,~assume that initial data~$(v_0,\r_0) \in H^1(\O)$ with
\begin{equation*}
  \dz{\nh \d v_0(x,y,z)}=0,~\textnormal{for all}~(x,y)\in M,
\end{equation*}we prove that the scaled Boussinesq equations~(\ref{eq:ptv})~strongly converge to the viscous primitive equations with density stratification~(\ref{eq:ptl})~as the aspect ration parameter~$\e$~goes to zero.

The following proposition is formally obtained by testing the scaled Boussinesq equations~(\ref{eq:ptv})~with~$(v,w,\r)$.~As for the rigorous justification for this proposition,~we refer to the work of Li-Titi\cite{lt2019}~and~Bardos \textit{et al.}\cite{ba2013}.

\begin{proposition}\label{po:vew}
Given a periodic function pair $(v_0,\r_0) \in H^1(\O)$ with
\begin{equation*}
  \dz{\nh \d v_0}=0~\textnormal{and}~w_0(x,y,z)=-\dk{\nh \d v_0(x,y,\xi)}.
\end{equation*}Suppose that~$(v_\e,w_\e,\r_\e)$~is a global weak solution of the system~(\ref{eq:ptv}),~satisfying the energy inequality (\ref{fl:vwr}),~and that~$(v,\r)$~is the unique global strong solution of the system~(\ref{eq:ptl}).~Then the following integral equality holds
\begin{flalign}
  &\xkh{\oo{\xkh{v_\e \d v+\e^2 w_\e w+\r_\e \r}}}(r)-\frac{\e^2}{2}\norm{w(r)}^2_2\nonumber\\
  &\qquad+\rt{\xkh{\n v_\e:\n v+\e^2\n w_\e \d \n w+\n \r_\e \d \n \r}}\nonumber\\
  &=\norm{v_0}^2_2+\rt{\zkh{-(u_\e \d \n)v_\e \d v-\e^2(u_\e \d \n w_\e)w-(u_\e \d \n \r_\e)\r}}\nonumber\\
  &\qquad+\frac{\e^2}{2}\norm{w_0}^2_2+\e^2\rt{\xkh{\dk{\p_t v(x,y,\xi,t)}} \d \nh W_\e}\nonumber\\
  &\qquad+\norm{\r_0}^2_2+\rt{\xkh{v_\e \d \p_t v+\r_\e \p_t \r+w_\e\r-\r_\e w}}, \label{fl:rew}
\end{flalign}for any~$r \in [0,\infty)$.
\end{proposition}

With the help of this proposition,~we can estimate the difference function~$(V_\e,W_\e,\g_\e)$.

\begin{proposition}\label{po:I1t}
Let~$(V_\e,W_\e,\g_\e)=(v_\e-v,w_\e-w,\r_\e-\r)$.~Under the same assumptions as in Proposition~\ref{po:vew},~the following estimate holds
\begin{equation*}
  \sup_{0 \leq s \leq t}\xkh{\norm{(V_\e,\e W_\e,\g_\e)}^2_2}(s)+\ds{\norm{\n(V_\e,\e W_\e,\g_\e)}^2_2} \leq \e^2 \mathcal{K}_1(t),
\end{equation*}for any~$t \in [0,\infty)$,~where
\begin{equation*}
  \mathcal{K}_1(t)=Ce^{C\eta^2_4(t)}\zkh{\eta_4(t)+\eta^2_4(t)+\xkh{\norm{v_0}^2_2+\e^2\norm{w_0}^2_2+\norm{\r_0}^2_2}^2}.
\end{equation*}Here~$C$~is a positive constant that does not depend on~$\e$.
\end{proposition}
\begin{proof}[Proof.]
Multiplying the first three equation in system~(\ref{eq:ptl})~by~$v_\e$,~$w_\e$~and~$\r_\e$~respectively,~integrating over~$\O\times(0,r)$,~then it follows from integration by parts that
\begin{flalign}
  &\rt{\xkh{v_\e \d \p_t v+\r_\e \p_t \r+\n v_\e:\n v+\n \r_\e \d \n \r}}\nonumber\\
  &\qquad=\rt{\zkh{w\r_\e-\r w_\e-(u \d \n)v \d v_\e-(u \d \n \r)\r_\e}}.\label{fl:udn}
\end{flalign}Replacing~$(v_\e,w_\e,\r_\e)$~with~$(v,w,\r)$,~a similar argument gives
\begin{flalign}
  &\frac{1}{2}\xkh{\norm{v(r)}^2_2+\norm{\r(r)}^2_2}+\dt{\xkh{\norm{\n v}^2_2+\norm{\n \r}^2_2}}\nonumber\\
  &\qquad=\frac{1}{2}\xkh{\norm{v_0}^2_2+\norm{\r_0}^2_2}.\label{fl:rtw}
\end{flalign}Thanks to Remark~\ref{re:lh},~the weak solution~$(v_\e,w_\e,\r_\e)$~of the system~(\ref{eq:ptv})~satisfies the following energy inequality
\begin{flalign}
  &\frac{1}{2}\xkh{\norm{v_\e(r)}^2_2+\e^2\norm{w_\e(r)}^2_2+\norm{\r_\e(r)}^2_2}\nonumber\\
  &\qquad\quad+\dt{\xkh{\norm{\n v_\e}^2_2+\e^2\norm{\n w_\e}^2_2+\norm{\n \r_\e}^2_2}}\nonumber\\
  &\qquad\leq\frac{1}{2}\xkh{\norm{v_0}^2_2+\e^2\norm{w_0}^2_2+\norm{\r_0}^2_2}.\label{fl:e22}
\end{flalign}
Subtracting the sum of~(\ref{fl:rew})~and~(\ref{fl:udn})~from the sum of~(\ref{fl:rtw})~and~(\ref{fl:e22}),~we have
\begin{flalign}
  &\frac{1}{2}\xkh{\norm{V_\e(r)}^2_2+\e^2\norm{W_\e(r)}^2_2+\norm{\g_\e(r)}^2_2}\nonumber\\
  &\qquad\quad+\dt{\xkh{\norm{\n V_\e}^2_2+\e^2\norm{\n W_\e}^2_2+\norm{\n \g_\e}^2_2}}\nonumber\\
  &\qquad\leq\rt{\zkh{(u_\e \d \n\r_\e)\r+(u \d \n\r)\r_\e}}\nonumber\\
  &\qquad\quad+\e^2\rt{\zkh{-\xkh{\dk{\p_t v(x,y,\xi,t)}} \d \nh W_\e-\n w \d \n W_\e}}\nonumber\\
  &\qquad\quad+\rt{\zkh{(u_\e \d \n)v_\e \d v+(u \d \n)v \d v_\e}}\nonumber\\
  &\qquad\quad+\e^2\rt{(u_\e \d \n w_\e)w}=:G_1+G_2+G_3+G_4.\label{fl:Ver}
\end{flalign}Firstly,~we estimate the integral term $G_1$ on the right-hand side of~(\ref{fl:Ver}). Using the H\"{o}lder inequality,~Lemma~\ref{le:phi} and Young inequality gives
\begin{flalign}
  G_1:&=\rt{\zkh{(u_\e \d \n\r_\e)\r+(u \d \n\r)\r_\e}}\nonumber\\
  &=\rt{\zkh{\xkh{V_\e \d \nh\g_\e}\r-(\p_z W_\e) \g_\e \r-W_\e \g_\e \p_z\r}}\nonumber\\
  &=\rt{\zkh{\xkh{V_\e \d \nh\g_\e}\r+(\nh \d V_\e) \g_\e \r}}\nonumber\\
  &\quad+\rt{\g_\e(\p_z\r)\xkh{\dk{(\nh \d V_\e)}}}\nonumber\\
  &=\rt{\xkh{|V_\e| |\nh\g_\e| |\r|+|\nh V_\e| |\g_\e| |\r|}}\nonumber\\
  &\quad+\mm{\xkh{\dz{|\nh V_\e|}}\xkh{\dz{|\g_\e||\p_z\r|}}}\nonumber\\
  &\leq C\dt{\norm{\n \r}^2_2\norm{\la \r}^2_2\xkh{\norm{V_\e}^2_2+\norm{\g_\e}^2_2}}\nonumber\\
  &\quad+\frac{1}{8}\dt{\xkh{\norm{\n V_\e}^2_2+\norm{\n \g_\e}^2_2}},\label{fl:ge2}
\end{flalign}note that the divergence-free condition, Sobolev embedding $H^1 \subset L^6$ and Poincar\'{e} inequality are used. By virtue of the H\"{o}lder inequality and Young inequality, we obtain
\begin{flalign}
  G_2:&=\e^2\rt{\zkh{-\xkh{\dk{\p_t v(x,y,\xi,t)}} \d \nh W_\e-\n w \d \n W_\e}}\nonumber\\
  &\leq C\e^2\dt{\xkh{\norm{\p_t v}^2_2+\norm{\nh w}^2_2+\norm{\p_z w}^2_2}}+\frac{1}{8}\dt{\e^2\norm{\n W_\e}^2_2}\nonumber\\
  &\leq C\e^2\dt{\zkh{\norm{\p_t v}^2_2+\oo{\xkh{\dk{\nh(\nh \d v)}}^2}}}\nonumber\\
  &\quad+C\e^2\dt{\norm{\nh v}^2_2}+\frac{1}{8}\dt{\e^2\norm{\n W_\e}^2_2}\nonumber\\
  &\leq C\e^2\dt{\xkh{\norm{\p_t v}^2_2+\norm{\la v}^2_2}}+\frac{1}{8}\dt{\e^2\norm{\n W_\e}^2_2}.~\label{fl:nWe}
\end{flalign}By the similar method as the integral term $G_1$ on the right-hand side of~(\ref{fl:Ver}),~the integral terms $G_3$ and $G_4$ on the right-hand side of~(\ref{fl:Ver}) can be estimated as
\begin{flalign}
  G_3:&=\rt{\zkh{(u_\e \d \n)v_\e \d v+(u \d \n)v \d v_\e}}\nonumber\\
  &\leq C\dt{\norm{\n v}^2_2 \norm{\la v}^2_2 \norm{V_\e}^2_2}+\frac{1}{8}\dt{\norm{\n V_\e}^2_2}\label{fl:nve}
\end{flalign}and
\begin{flalign}
  G_4:&=\e^2\rt{(u_\e \d \n w_\e)w}\nonumber\\
  &\leq C\e^2\dt{\xkh{\norm{v_\e}^2_2 \norm{\n v_\e}^2_2+\norm{\n v}^2_2 \norm{\la v}^2_2+\e^4\norm{w_\e}^2_2 \norm{\n w_\e}^2_2}}\nonumber\\
  &\quad+\frac{1}{8}\dt{\xkh{\norm{\n V_\e}^2_2+\e^2\norm{\n W_\e}^2_2}},\label{fl:e2r}
\end{flalign}respectively.

Adding~(\ref{fl:ge2}),~(\ref{fl:nWe}),~(\ref{fl:nve}),~and~(\ref{fl:e2r})~yields
\begin{flalign*}
  h(t):&=\xkh{\norm{(V_\e,\g_\e)}^2_2+\e^2\norm{W_\e}^2_2}(t)+\ds{\xkh{\norm{\n(V_\e,\g_\e)}^2_2+\e^2\norm{\n W_\e}^2_2}}\\
  &\leq C\ds{\norm{\n v}^2_2 \norm{\la v}^2_2 \norm{V_\e}^2_2}+C\e^2\ds{\xkh{\norm{\p_t v}^2_2+\norm{\la v}^2_2}}\\
  &\quad+C\e^2\ds{\xkh{\norm{v_\e}^2_2 \norm{\n v_\e}^2_2+\norm{\n v}^2_2 \norm{\la v}^2_2+\e^4\norm{w_\e}^2_2 \norm{\n w_\e}^2_2}}\\
  &\quad+C\ds{\norm{\n \r}^2_2 \norm{\la \r}^2_2 \xkh{\norm{V_\e}^2_2+\norm{\g_\e}^2_2}}=:H(t),
\end{flalign*}for a.e.~$t \in [0,\infty)$.~Taking the derivative of~$H(t)$~with respect to~$t$,~we obtain
\begin{flalign*}
  H'(t)&\leq C\e^2\xkh{\norm{v_\e}^2_2 \norm{\n v_\e}^2_2+\norm{\n v}^2_2 \norm{\la v}^2_2+\e^4\norm{w_\e}^2_2 \norm{\n w_\e}^2_2}\\
  &\quad+C\xkh{\norm{\n v}^2_2 \norm{\la v}^2_2+\norm{\n \r}^2_2 \norm{\la \r}^2_2} \xkh{\norm{V_\e}^2_2+\norm{\g_\e}^2_2}+C\e^2\xkh{\norm{\p_t v}^2_2+\norm{\la v}^2_2}\\
  &\leq C\xkh{\norm{\n v}^2_2 \norm{\la v}^2_2+\norm{\n \r}^2_2 \norm{\la \r}^2_2}H(t)+C\e^2\xkh{\norm{\p_t v}^2_2+\norm{\la v}^2_2}\\
  &\quad+C\e^2\xkh{\norm{v_\e}^2_2 \norm{\n v_\e}^2_2+\norm{\n v}^2_2 \norm{\la v}^2_2+\e^4\norm{w_\e}^2_2 \norm{\n w_\e}^2_2}.
\end{flalign*}Applying the Gronwall inequality to the above inequality,~it follows from~(\ref{fl:e4t})~and~(\ref{fl:e22})~that
\begin{flalign*}
  h(t):&=\xkh{\norm{(V_\e,\g_\e)}^2_2+\e^2\norm{W_\e}^2_2}(t)+\ds{\xkh{\norm{\n(V_\e,\g_\e)}^2_2+\e^2\norm{\n W_\e}^2_2}}\\
  &\leq C\e^2\exp\dkh{C\ds{\xkh{\norm{\n v}^2_2 \norm{\la v}^2_2+\norm{\n \r}^2_2 \norm{\la \r}^2_2}}}\\
  &\quad\times\ds{\xkh{\norm{\p_t v}^2_2+\norm{\la v}^2_2+\norm{v_\e}^2_2 \norm{\n v_\e}^2_2+\norm{\n v}^2_2 \norm{\la v}^2_2+\e^4\norm{w_\e}^2_2 \norm{\n w_\e}^2_2}}\\
  &\leq C\e^2 e^{C\eta^2_4(t)}\zkh{\eta_4(t)+\eta^2_4(t)+\xkh{\norm{v_0}^2_2+\e^2\norm{w_0}^2_2+\norm{\r_0}^2_2}^2},
\end{flalign*}note that the fact that~$H(0)=0$.~This completes the proof.
\end{proof}

Based on Proposition \ref{po:I1t}, the proof of Theorem~\ref{th:I1t} is shown below.
\begin{proof}[Proof of Theorem~\ref{th:I1t}.]
For any~$T>0$,~according to Proposition~\ref{po:I1t},~the following estimate holds
\begin{equation*}
  \sup_{0 \leq t \leq T}\xkh{\norm{(V_\e,\e W_\e,\g_\e)}^2_2}(t)
  +\int^T_0{\norm{\n(V_\e,\e W_\e,\g_\e)}^2_2}dt \leq \e^2\widetilde{\mathcal{K}_1}(T),
\end{equation*}where
\begin{equation*}
  \widetilde{\mathcal{K}_1}(T)=Ce^{C\eta^2_4(T)}\zkh{\eta_4(T)+\eta^2_4(T)+\xkh{\norm{v_0}^2_2+\norm{w_0}^2_2+\norm{\r_0}^2_2}^2}.
\end{equation*}Here~$C$~is a positive constant that does not depend on~$\e$.~It deduces from the above estimate that
\begin{gather*}
  (v_\e,\e w_\e,\r_\e) \rightarrow (v,0,\r),~in~L^{\infty}\xkh{[0,T];L^2(\O)},\\
  (\n v_\e,\e \n w_\e,\n \r_\e,w_\e) \rightarrow (\n v,0,\n \r,w),~in~L^2\xkh{[0,T];L^2(\O)}.
\end{gather*}Obviously,~the rate of convergence is of the order~$O(\e)$.~The theorem is thus proved.
\end{proof}

\section{Strong convergence for~$H^2$~initial data}
In this section,~assume that initial data~$(v_0,\r_0) \in H^2(\O)$, where initial velocity $v_0$ satisfies
\begin{equation*}
  \dz{\nh \d v_0(x,y,z)}=0,~\textnormal{for all}~(x,y)\in M,
\end{equation*}we prove that the scaled Boussinesq equations~(\ref{eq:ptv})~strongly converge to the viscous primitive equations with density stratification~(\ref{eq:ptl})~as the aspect ration parameter~$\e$~goes to zero.~In this case,~there is a unique local strong solution~$(v_\e,w_\e,\r_\e)$~to the system~(\ref{eq:ptv}),~subject to the boundary and initial conditions~(\ref{ga:are})-(\ref{ga:vew})~and symmetry condition~(\ref{eq:eve}).~Denote by~$T^*_\e$~the maximal existence time of this local strong solution.

Let
\begin{gather*}
  (U_\e,\g_\e,P_\e)=(V_\e,W_\e,\g_\e,P_\e), \\
  (V_\e,W_\e,\g_\e,P_\e)=(v_\e-v,w_\e-w,\r_\e-\r,p_\e-p).
\end{gather*}We subtract the system~(\ref{eq:ptl})~from the system~(\ref{eq:ptv})~and then lead to the following system
\begin{flalign}
  &\p_t V_\e-\la V_\e+(U_\e \d \n)V_\e+(u \d \n)V_\e+(U_\e \d \n)v+\nh P_\e=0, \label{fl:Ve0}\\
  &\e^2(\p_t W_\e-\la W_\e+U_\e \d \n W_\e+U_\e \d \n w+u \d \n W_\e)+\p_z P_\e\nonumber\\
  &+\g_\e+\e^2(\p_t w-\la w+u \d \n w)=0, \label{fl:dnw}\\
  &\p_t \g_\e-\la \g_\e+U_\e \d \n \g_\e+U_\e \d \n \r+u \d \n \g_\e-W_\e=0, \label{fl:nge}\\
  &\nh \d V_\e+\p_z W_\e=0, \label{fl:zWe}
\end{flalign}defined on~$\O \times (0,T^*_\e)$.

\begin{proposition}\label{po:I1T}
Suppose that~$(v_0,\r_0) \in H^2(\O)$,~with~$\dz{\nh \d v_0}=0$.~Then the system~(\ref{fl:Ve0})-(\ref{fl:zWe})~has the following basic energy estimate
\begin{equation*}
  \sup_{0 \leq s \leq t}\xkh{\norm{(V_\e,\e W_\e,\g_\e)}^2_2}(s)+\ds{\norm{\n(V_\e,\e W_\e,\g_\e)}^2_2} \leq \e^2\mathcal{K}_1(t),
\end{equation*}for any~$t \in [0,T^*_\e)$,~where
\begin{equation*}
  \mathcal{K}_1(t)=Ce^{C\eta^2_4(t)}\zkh{\eta_4(t)+\eta^2_4(t)+\xkh{\norm{v_0}^2_2+\e^2\norm{w_0}^2_2+\norm{\r_0}^2_2}^2}.
\end{equation*}Here~$C$~is a positive constant that does not depend on~$\e$.
\end{proposition}

It is important to note that the Proposition~\ref{po:I1T}~is a direct consequence of Proposition~\ref{po:I1t}.~Moreover,~the basic energy estimate on the system~(\ref{fl:Ve0})-(\ref{fl:zWe})~can also be obtained by the energy method.~The strong solution~$(v_\e,w_\e,\r_\e)$~to the system~(\ref{eq:ptv})~is local,~so is this basic energy estimate.~In order to obtain the first order energy estimate for the system~(\ref{fl:Ve0})-(\ref{fl:zWe}),~we need to perform the second order energy estimate on the system~(\ref{eq:ptl}).

\begin{proposition}\label{po:e5t}
Suppose that~$(v_0,\r_0) \in H^2(\O)$,~with~$\dz{\nh \d v_0}=0$.~Then the system~(\ref{eq:ptl})~has the following second order energy estimate
\begin{equation*}
  \sup_{0 \leq s \leq t}\xkh{\norm{\la v}^2_2+\norm{\la \r}^2_2}(s)
  +\ds{\xkh{\norm{\n \p_t v}^2_2+\norm{\n \la v}^2_2+\norm{\n \p_t \r}^2_2+\norm{\n \la \r}^2_2}} \leq  \eta_5(t),
\end{equation*}for any~$t \in [0,\infty)$,~where
\begin{equation*}
  \eta_5(t)=e^{C(t+2)\xkh{\eta^2_4(t)+1}}\zkh{\norm{v_0}^2_{H^2}+\norm{\r_0}^2_{H^2}}.
\end{equation*}
\end{proposition}
\begin{proof}[Proof.]
Taking the~$L^2(\O)$~inner product of the equation~(\ref{fl:dkr})~and~(\ref{fl:lar})~with~$\la\xkh{\la v-\p_t v}$~and~$\la$ $\xkh{\la \r-\p_t \r}$~respectively,~then it deduces from integration by parts that
\begin{flalign}
  \frac{d}{dt}&\xkh{\norm{\la v}^2_2+\norm{\la \r}^2_2}+\norm{\n \p_t v}^2_2+\norm{\n \la v}^2_2+\norm{\n \p_t \r}^2_2+\norm{\n \la \r}^2_2\nonumber\\
  &=\oo{\n\xkh{\dk{\nh \d v(x,y,\xi,t)}} \d \n\xkh{\la \r-\p_t \r}}\nonumber\\
  &\quad+\oo{\n\xkh{\dk{\nh \r(x,y,\xi,t)}}:\n\xkh{\p_t v-\la v}}\nonumber\\
  &\quad+\oo{\n\zkh{v \d \nh \r-\xkh{\dk{\nh \d v(x,y,\xi,t)}} \p_z \r} \d \n\xkh{\la \r-\p_t \r}}\nonumber\\
  &\quad+\oo{\n\zkh{\xkh{v \d \nh}v-\xkh{\dk{\nh \d v(x,y,\xi,t)}} \p_z v}:\n\xkh{\la v-\p_t v}}\nonumber\\
  &=:G_1+G_2+G_3+G_4.\label{fl:vpt}
\end{flalign}In order to estimate the first integral term~$G_1$~on the right-hand side of~(\ref{fl:vpt}),~we use the H\"{o}lder inequality and Young inequality to reach
\begin{flalign}
  G_1:&=\oo{\n\xkh{\dk{\nh \d v(x,y,\xi,t)}} \d \n\xkh{\la \r-\p_t \r}}\nonumber\\
  &=\oo{\nh\xkh{\dk{\nh \d v(x,y,\xi,t)}} \d \nh\xkh{\la \r-\p_t \r}}\nonumber\\
  &\quad+\oo{(\nh \d v)\xkh{\p_z \la \r-\p_z \p_t \r}}\nonumber\\
  &=\oo{\zkh{\dk{\nh \d \p_i v(x,y,\xi,t)}} \xkh{\p_i \la \r-\p_i \p_t \r}}\nonumber\\
  &\quad+\oo{(\nh \d v)\xkh{\p_z \la \r-\p_z \p_t \r}}\nonumber\\
  &\leq\mm{\xkh{\dz{|\p_i \nh v|}} \xkh{\dz{\xkh{|\p_i \p_t \r|+|\p_i \la \r|}}}}\nonumber\\
  &\quad+\oo{|\nh v| \xkh{|\p_z \p_t \r|+|\p_z \la \r|}}\nonumber\\
  &\leq C\xkh{\norm{\n v}^2_2+\norm{\la v}^2_2}
  +\frac{1}{8}\xkh{\norm{\n \p_t \r}^2_2+\norm{\n \la \r}^2_2}.\label{fl:G1r}
\end{flalign}For the second integral term~$G_2$~on the right-hand side of~(\ref{fl:vpt}),~using the same method as the integral term~$G_1$~gives
\begin{flalign}
  G_2:&=\oo{\n\xkh{\dk{\nh \r(x,y,\xi,t)}}:\n\xkh{\p_t v-\la v}}\nonumber\\
  &\leq C\xkh{\norm{\n \r}^2_2+\norm{\la \r}^2_2}
  +\frac{1}{8}\xkh{\norm{\n \p_t v}^2_2+\norm{\n \la v}^2_2}.\label{fl:G2v}
\end{flalign}To obtain the estimate on the third integral term~$G_3$~on the right-hand side of~(\ref{fl:vpt}),~applying the Lemma~\ref{le:psi}~and Young inequality yields
\begin{flalign}
  G_3:&=\oo{\n\zkh{v \d \nh \r-\xkh{\dk{\nh \d v(x,y,\xi,t)}} \p_z \r} \d \n\xkh{\la \r-\p_t \r}}\nonumber\\
  &=\oo{\n\xkh{u \d \n \r} \d \n\xkh{\la \r-\p_t \r}}\nonumber\\
  &=\oo{\xkh{\p_i u \d \n \r+u \d \p_i \n \r} \xkh{\p_i \la \r-\p_i \p_t \r}}\nonumber\\
  &\leq C\norm{\p_i \n v}^{1/2}_2 \norm{\p_i \la v}^{1/2}_2 \norm{\n \r}^{1/2}_2 \norm{\la \r}^{1/2}_2
  \xkh{\norm{\p_i \p_t \r}_2+\norm{\p_i \la \r}_2}\nonumber\\
  &\quad+C\norm{\n v}^{1/2}_2 \norm{\la v}^{1/2}_2 \norm{\p_i \n \r}^{1/2}_2 \norm{\p_i \la \r}^{1/2}_2
  \xkh{\norm{\p_i \p_t \r}_2+\norm{\p_i \la \r}_2}\nonumber\\
  &\leq C\xkh{\norm{\n v}^2_2\norm{\la v}^2_2+\norm{\n \r}^2_2\norm{\la \r}^2_2} \xkh{\norm{\la v}^2_2+\norm{\la \r}^2_2}\nonumber\\
  &\quad+\frac{1}{8}\xkh{\norm{\n \la v}^2_2+\norm{\n \p_t \r}^2_2+\norm{\n \la \r}^2_2}.\label{fl:G3r}
\end{flalign}As for the last integral term~$G_4$~on the right-hand side of~(\ref{fl:vpt}),~a similar argument~as the integral term $G_3$~leads to
\begin{flalign}
  G_4:&=\oo{\n\zkh{\xkh{v \d \nh}v-\xkh{\dk{\nh \d v(x,y,\xi,t)}} \p_z v}:\n\xkh{\la v-\p_t v}}\nonumber\\
  &=\oo{\n\zkh{(u \d \n)v} : \n\xkh{\la v-\p_t v}}\nonumber\\
  &\leq C\norm{\n v}^2_2\norm{\la v}^2_2\norm{\la v}^2_2+\frac{1}{8}\xkh{\norm{\n \p_t v}^2_2+\norm{\n \la v}^2_2}.\label{fl:G4v}
\end{flalign}Substituting~(\ref{fl:G1r})-(\ref{fl:G4v})~into~(\ref{fl:vpt}),~we have
\begin{flalign*}
  \frac{d}{dt}&\xkh{\norm{\la v}^2_2+\norm{\la \r}^2_2}+\frac{1}{2}\xkh{\norm{\n \p_t v}^2_2+\norm{\n \la v}^2_2+\norm{\n \p_t \r}^2_2+\norm{\n \la \r}^2_2}\nonumber\\
  &\leq C\xkh{\norm{\n v}^2_2\norm{\la v}^2_2+\norm{\n \r}^2_2\norm{\la \r}^2_2+1} \xkh{\norm{\la v}^2_2+\norm{\la \r}^2_2}.
\end{flalign*}Thanks to the Gronwall inequality,~it follows from~(\ref{fl:e4t})~that
\begin{flalign*}
  &\xkh{\norm{\la v}^2_2+\norm{\la \r}^2_2}(t)+\ds{\xkh{\norm{\n \p_t v}^2_2+\norm{\n \la v}^2_2+\norm{\n \p_t \r}^2_2+\norm{\n \la \r}^2_2}}\nonumber\\
  &\qquad\leq \exp\dkh{C\ds{\xkh{\norm{\n v}^2_2\norm{\la v}^2_2+\norm{\n \r}^2_2\norm{\la \r}^2_2+1}}}\xkh{\norm{\la v_0}^2_2+\norm{\la \r_0}^2_2}\nonumber\\
  &\qquad\leq Ce^{C(t+2)\xkh{\eta^2_4(t)+1}}\zkh{\norm{v_0}^2_{H^2}+\norm{\r_0}^2_{H^2}}.
\end{flalign*}~The proof is completed.
\end{proof}

With the help of Proposition \ref{po:e5t}, we can perform the first order energy estimate on the system (\ref{fl:Ve0})-(\ref{fl:zWe}).
\begin{proposition}\label{po:I2t}
Suppose that~$(v_0,\r_0) \in H^2(\O)$,~with~$\dz{\nh \d v_0}=0$.~Then there exists a small positive constant~$\beta_0$~such that the system~(\ref{fl:Ve0})-(\ref{fl:zWe})~has the following first order energy estimate
\begin{equation*}
  \sup_{0 \leq s \leq t}\xkh{\norm{\n(V_\e,\e W_\e,\g_\e)}^2_2}(s)
  +\ds{\norm{\la(V_\e,\e W_\e,\g_\e)}^2_2} \leq \e^2 \mathcal{K}_2(t),
\end{equation*}for any~$t \in [0,T^*_\e)$,~provided that
\begin{equation*}
  \sup_{0 \leq s \leq t}\xkh{\norm{\n (V_\e,\g_\e)}^2_2+\e^2\norm{\n W_\e}^2_2}(s) \leq \beta^2_0,
\end{equation*}where
\begin{equation*}
  \mathcal{K}_2(t)=Ce^{C(1+\e^4)\eta^2_5(t)} \zkh{\eta_5(t)+\eta^2_5(t)}.
\end{equation*}Here~$C$~is a positive constant that does not depend on~$\e$.
\end{proposition}
\begin{proof}[Proof.]
Multiplying the first three equation~in system~(\ref{fl:Ve0})-(\ref{fl:zWe})~by~$-\la V_\e$,~$-\la W_\e$~and~$-\la \g_\e$~respectively,~then integrating over~$\O\times(0,r)$,~and finally integrating by parts lead to
\begin{flalign}
  \frac{1}{2}\frac{d}{dt}&\xkh{\norm{\n (V_\e,\g_\e)}^2_2+\e^2\norm{\n W_\e}^2_2}
  +\norm{\la(V_\e,\g_\e)}^2_2+\e^2\norm{\la W_\e}^2_2\nonumber\\
  &=\oo{(U_\e \d \n \g_\e+u \d \n \g_\e+U_\e \d \n \r)\la \g_\e}\nonumber\\
  &\quad+\e^2\oo{\xkh{\p_t w-\la w+u \d \n w}\la W_\e}\nonumber\\
  &\quad+\e^2\oo{\xkh{U_\e \d \n W_\e+u \d \n W_\e+U_\e \d \n w}\la W_\e}\nonumber\\
  &\quad+\oo{\zkh{(U_\e \d \n)V_\e+(u \d \n)V_\e+(U_\e \d \n)v} \d \la V_\e}\nonumber\\
  &=:G_1+G_2+G_3+G_4.\label{fl:Veg}
\end{flalign}For the first integral term~$G_1$~on the right-hand side of~(\ref{fl:Veg}),~we apply the Lemma~\ref{le:psi} and Poincar\'{e} inequality and Young inequality to obtain
\begin{flalign}
  G_1:&=\oo{(U_\e \d \n \g_\e+u \d \n \g_\e+U_\e \d \n \r)\la \g_\e}\nonumber\\
  &\leq\frac{5}{64}\xkh{\norm{\la V_\e}^2_2+\norm{\la \g_\e}^2_2}+C\xkh{\norm{\n (V_\e,\g_\e)}^2_2+\e^2\norm{\n W_\e}^2_2}\nonumber\\
  &\quad\times\zkh{\xkh{\norm{\la V_\e}^2_2+\norm{\la \g_\e}^2_2}+\norm{\la \r}^2_2 \norm{\n \la \r}^2_2+\norm{\la v}^2_2 \norm{\n \la v}^2_2}.\label{fl:221}
\end{flalign}Thanks to the H\"{o}lder inequality, Lemma~\ref{le:psi} and Young inequality, we have
\begin{flalign}
  G_2:&=\e^2\oo{\xkh{\p_t w-\la w+u \d \n w}\la W_\e}\nonumber\\
  &\leq \e^2C\xkh{\norm{\la v}^2_2 \norm{\n \la v}^2_2+\norm{\n \p_t v}^2_2+\norm{\n \la v}^2_2}+\frac{5}{64}\e^2\norm{\la W_\e}^2_2.\label{fl:law}
\end{flalign}where the divergence-free condition is used. Using the similar method as the integral term~$G_1$~on the right-hand side of~(\ref{fl:Veg}),~the integral term~$G_3$ and~$G_4$~can be bounded as
\begin{flalign}
  G_3:&=\e^2\oo{\xkh{U_\e \d \n W_\e+u \d \n W_\e+U_\e \d \n w}\la W_\e}\nonumber\\
  &\leq\frac{5}{64}\xkh{\norm{\la V_\e}^2_2+\e^2\norm{\la W_\e}^2_2}+C\xkh{\norm{\n (V_\e,\g_\e)}^2_2+\e^2\norm{\n W_\e}^2_2}\nonumber\\
  &\quad\times\zkh{\xkh{\norm{\la V_\e}^2_2+\e^2\norm{\la W_\e}^2_2}+(1+\e^4)\norm{\la v}^2_2 \norm{\n \la v}^2_2}\label{fl:laV}
\end{flalign}and
\begin{flalign}
  G_4:&=\oo{\zkh{(U_\e \d \n)V_\e+(u \d \n)V_\e+(U_\e \d \n)v} \d \la V_\e}\nonumber\\
  &\leq C\xkh{\norm{\n V_\e}^2_2 \norm{\la V_\e}^2_2+\norm{\la v}^2_2 \norm{\n \la v}^2_2 \norm{\n V_\e}^2_2}+\frac{5}{64}\norm{\la V_\e}^2_2\nonumber\\
  &\leq C\norm{\n V_\e}^2_2\xkh{\norm{\la V_\e}^2_2+\norm{\la v}^2_2 \norm{\n \la v}^2_2 }+\frac{5}{64}\norm{\la V_\e}^2_2,\label{fl:nla}
\end{flalign}respectively.

Combining the estimates for~(\ref{fl:221}),~(\ref{fl:law}),~(\ref{fl:laV}),~and~(\ref{fl:nla}),~we obtain
\begin{flalign}
  \frac{1}{2}\frac{d}{dt}&\xkh{\norm{\n (V_\e,\g_\e)}^2_2+\e^2\norm{\n W_\e}^2_2}+\frac{11}{16}\xkh{\norm{\la(V_\e,\g_\e)}^2_2+\e^2\norm{\la W_\e}^2_2}\nonumber\\
  &\leq C_\delta\xkh{\norm{\n (V_\e,\g_\e)}^2_2+\e^2\norm{\n W_\e}^2_2}\Big[\xkh{\norm{\la (V_\e,\g_\e)}^2_2+\e^2\norm{\la W_\e}^2_2}\nonumber\\
  &\quad+\norm{\la \r}^2_2 \norm{\n \la \r}^2_2+(1+\e^4)\norm{\la v}^2_2 \norm{\n \la v}^2_2\Big]\nonumber\\
  &\quad+\e^2 C_\delta\xkh{\norm{\la v}^2_2 \norm{\n \la v}^2_2+\norm{\n \p_t v}^2_2+\norm{\n \la v}^2_2}.\label{fl:ddt}
\end{flalign}Using the assumption given by the proposition
\begin{equation*}
  \sup_{0 \leq s \leq t}\xkh{\norm{\n (V_\e,\g_\e)}^2_2+\e^2\norm{\n W_\e}^2_2}(s) \leq \beta^2_0,
\end{equation*}and choosing~$\beta_0=\sqrt{\frac{3}{16C_\delta}}$,~it deduces from inequality~(\ref{fl:ddt})~that
\begin{flalign*}
  \frac{d}{dt}&\xkh{\norm{\n (V_\e,\g_\e)}^2_2+\e^2\norm{\n W_\e}^2_2}+\xkh{\norm{\la(V_\e,\g_\e)}^2_2+\e^2\norm{\la W_\e}^2_2}\\
  &\leq C_\delta\zkh{\norm{\la \r}^2_2 \norm{\n \la \r}^2_2+(1+\e^4)\norm{\la v}^2_2 \norm{\n \la v}^2_2}
  \xkh{\norm{\n (V_\e,\g_\e)}^2_2+\e^2\norm{\n W_\e}^2_2}\\
  &\quad+\e^2 C_\delta\xkh{\norm{\la v}^2_2 \norm{\n \la v}^2_2+\norm{\n \p_t v}^2_2+\norm{\n \la v}^2_2}.
\end{flalign*}Applying the Gronwall inequality to the above inequality,~it follows from~Propositin~\ref{po:e5t}~that
\begin{flalign*}
  &\xkh{\norm{\n (V_\e,\g_\e)}^2_2+\e^2\norm{\n W_\e}^2_2}(t)+\ds{\xkh{\norm{\la(V_\e,\g_\e)}^2_2+\e^2\norm{\la W_\e}^2_2}}\\
  &\qquad\leq \e^2 C_\delta\exp\dkh{C_\delta\ds{\zkh{\norm{\la \r}^2_2 \norm{\n \la \r}^2_2+(1+\e^4)\norm{\la v}^2_2 \norm{\n \la v}^2_2}}}\\
  &\qquad\quad\times\ds{\xkh{\norm{\la v}^2_2 \norm{\n \la v}^2_2+\norm{\n \p_t v}^2_2+\norm{\n \la v}^2_2}}\\
  &\qquad\leq \e^2 C_\delta e^{C_\delta(1+\e^4)\eta^2_5(t)} \zkh{\eta_5(t)+\eta^2_5(t)},
\end{flalign*}note that the fact that~$(V_\e,W_\e,\g_\e)|_{t=0}=0$.~This completes the proof.
\end{proof}

\begin{proposition}\label{po:I12}
Let $T^*_\e$ be the maximal existence time of the strong solution~$(v_\e,w_\e,\r_\e)$~to the system~(\ref{eq:ptv})~corresponding to boundary and initial conditions~(\ref{ga:are})-(\ref{ga:vew})~and symmetry condition~(\ref{eq:eve}).~Then,~for any~$T>0$,~there exists a small positive constant~$\e(T)=\frac{3\beta_0}{4\sqrt{\widetilde{\mathcal{K}_2}(T)}}$~such that~$T^*_\e>T$~provided that~$\e \in (0,\e(T))$.~Furthermore,~the system~(\ref{fl:Ve0})-(\ref{fl:zWe})~has the following estimate
\begin{equation*}
  \sup_{0 \leq t \leq T}\xkh{\norm{(V_\e,\e W_\e,\g_\e)}^2_{H^1}}(t)
  +\int^{T}_0{\norm{\n(V_\e,\e W_\e,\g_\e)}^2_{H^1}}dt \leq \e^2\xkh{\widetilde{\mathcal{K}_1}(T)+\widetilde{\mathcal{K}_2}(T)},
\end{equation*}where
\begin{equation*}
  \widetilde{\mathcal{K}_1}(T)=Ce^{C\eta^2_4(T)}\zkh{\eta_4(T)+\eta^2_4(T)+\xkh{\norm{v_0}^2_2+\norm{w_0}^2_2+\norm{\r_0}^2_2}^2},
\end{equation*}and
\begin{equation*}
  \widetilde{\mathcal{K}_2}(T)=C'e^{C'\eta^2_5(T)} \zkh{\eta_5(T)+\eta^2_5(T)}.
\end{equation*}
Here both~$C$~and~$C'$~are positive constants that do not depend on $\e$.
\end{proposition}

The proof of Proposition~\ref{po:I12}~is similar to that of Proposition 4.3 given in Pu-Zhou\cite{pz2021}~and so is omitted.~Based on Proposition~\ref{po:I12},~we give the proof of Theorem~\ref{th:I12}.

\begin{proof}[Proof of Theorem~\ref{th:I12}.]
For any~$T>0$,~thanks to Proposition~\ref{po:I12},~there exists a small positive constant~$\e(T)=\frac{3\beta_0}{4\sqrt{\widetilde{\mathcal{K}_2}(T)}}$~such that~$T^*_\e>T$~provided that~$\e \in (0,\e(T))$,~which implies that the system~(\ref{eq:ptv})~corresponding to boundary and initial conditions~(\ref{ga:are})-(\ref{ga:vew})~and symmetry condition~(\ref{eq:eve})~has a unique strong solution~$(v_\e,w_\e,\r_\e)$~on the time interval~$[0,T]$ for all $\e$ $\in$ $(0,\e(T))$.~Moreover,~the following estimate holds
\begin{flalign*}
  \sup_{0 \leq t \leq T}\xkh{\norm{(V_\e,\e W_\e,\g_\e)}^2_{H^1}}(t)
  +\int^{T}_0{\norm{\n(V_\e,\e W_\e,\g_\e)}^2_{H^1}}dt &\leq \e^2\xkh{\widetilde{\mathcal{K}_1}(T)+\widetilde{\mathcal{K}_2}(T)}\\
  &=:\e^2\widetilde{\mathcal{K}_3}(T),
\end{flalign*}where~$\widetilde{\mathcal{K}_3}(t)$~is a nonnegative continuously increasing function that does not depend on~$\e$.~Finally,~it is clear that the strong convergences stated in Theorem~\ref{th:I12}~are the direct consequences of the above estimate.~The theorem is thus proved.
\end{proof}

\section*{Acknowledgments}
The work of X. Pu was supported in part by the National Natural Science Foundation of China (No. 11871172) and the Natural Science Foundation of Guangdong Province of China (No. 2019A1515012000). The work of W. Zhou was supported by the Innovation Research for the Postgraduates of Guangzhou University (No. 2021GDJC-D09).

\section*{Appendix}
\setcounter{lemma}{0}
\renewcommand{\thelemma}{A.\arabic{lemma}}

In this appendix,~we present some Ladyzhenskaya-type inequalities in three dimensions for a class of integrals,~which are frequently used throughout the paper.
\begin{lemma}(\!\!\cite{ct2003})\label{le:phi}
The following inequalities hold
\begin{flalign*}
  \mm{&\xkh{\dz{\varphi(x,y,z)}} \xkh{\dz{\psi(x,y,z)\phi(x,y,z)}}}\\
  &\leq C\norm{\varphi}^{1/2}_2 \xkh{\norm{\varphi}^{1/2}_2+\norm{\nh\varphi}^{1/2}_2}
  \norm{\psi}^{1/2}_2 \xkh{\norm{\psi}^{1/2}_2+\norm{\nh\psi}^{1/2}_2} \norm{\phi}_2,
\end{flalign*}
\begin{flalign*}
  \mm{&\xkh{\dz{\varphi(x,y,z)}} \xkh{\dz{\psi(x,y,z)\phi(x,y,z)}}}\\
  &\leq C\norm{\psi}^{1/2}_2 \xkh{\norm{\psi}^{1/2}_2+\norm{\nh\psi}^{1/2}_2}
  \norm{\phi}^{1/2}_2 \xkh{\norm{\phi}^{1/2}_2+\norm{\nh\phi}^{1/2}_2} \norm{\varphi}_2,
\end{flalign*}for every~$\varphi,\psi,\phi$~such that the right-hand sides make sense and are finite,~where~$C$~is a positive constant.
\end{lemma}

\begin{lemma}(\!\!\cite{lt2019})\label{le:psi}
Let~$\varphi=(\varphi_1,\varphi_2,\varphi_3)$,~$\psi$~and~$\phi$~be periodic functions in~$\O$.~Denote by~$\varphi_h=(\varphi_1,\varphi_2)$~the horizontal components of the function~$\varphi$.~There exists a positive constant~$C$~such that the following estimate holds
\begin{equation*}
  \abs{\oo{\xkh{\varphi \d \n\psi}\phi}} \leq C \norm{\n \varphi_h}^{1/2}_2 \norm{\la \varphi_h}^{1/2}_2
  \norm{\n \psi}^{1/2}_2 \norm{\la \psi}^{1/2}_2 \norm{\phi}_2,
\end{equation*}provided that~$\varphi \in H^1(\O)$,~with~$\n \d \varphi=0$~in~$\O$,~$\oo{\varphi}=0$,~and~$\varphi_3|_{z=0}=0$,~$\n\psi \in H^1(\O)$~and~$\phi \in L^2(\O)$.
\end{lemma}


\newpage
\small
\begin{thebibliography}{99}
\bibitem{pa2001}P.~Az\'{e}rad,~F.~Guill\'{e}n,~Mathematical justification of the hydrostatic approximation in the primitive equations of geophysical fluid dynamics,~\emph{SIAM J. Math. Anal.},~33~(2001)~847-859.
\bibitem{ba2013}C.~Bardos,~M.C.~Lopes Filho,~D.~Niu,~H.J.~Nussenzveig Lopes,~E.S.~Titi,~Stability of two-dimensional viscous incompressible flows under three-dimensional perturbations and inviscid symmetry breaking,~\emph{SIAM J. Math. Anal.},~45~(2013)~1871-1885.
\bibitem{db2003}D.~Bresch,~F.~Guill\'{e}n-Gonz\'{a}lez,~N.~Masmoudi,~M.A.~Rodr\'{\i}guez-Bellido,~On the uniqueness of weak solutions of the two-dimensional primitive equations,~\emph{Differ. Integral Equ.},~16~(2003)~77-94.
\bibitem{cc2015}C.~Cao,~S.~Ibrahim,~K.~Nakanishi,~E.S.~Titi,~Finite-time blowup for the inviscid primitive equations of oceanic and atmospheric dynamics,~\emph{Commun. Math. Phys.},~337~(2015)~473-482.
\bibitem{cc2014}C.~Cao,~J.~Li,~E.S.~Titi,~Local and global well-posedness of strong solutions to the 3D primitive equations with vertical eddy diffusivity,~\emph{Arch. Ration. Mech. Anal.},~214~(2014) 35-76.
\bibitem{jl2014}C.~Cao,~J.~Li,~E.S.~Titi,~Global well-posedness of strong solutions to the 3D primitive equations with horizontal eddy diffusivity,~\emph{J. Differ. Equ.},~257~(2014)~4108-4132.
\bibitem{es2016}C.~Cao,~J.~Li,~E.S.~Titi,~Global well-posedness of the 3D primitive equations with only horizontal viscosity and diffusivity,~\emph{Commun. Pure Appl. Math.},~69~(2016)~1492-1531.
\bibitem{cc2017}C.~Cao,~J.~Li,~E.S.~Titi,~Strong solutions to the 3D primitive equations with horizontal dissipation:~near~$H^1$~initial data,~\emph{J. Funct. Anal.},~272~(2017)~4606-4641.
\bibitem{es2020}C.~Cao,~J.~Li,~E.S.~Titi,~Global well-posedness of the 3D primitive equations with horizontal viscosity and vertical diffusivity,~\emph{Phys. D},~412~(2020)~132606,~25~pp.
\bibitem{ct2007}C.~Cao,~E.S.~Titi,~Global well-posedness of the three-dimensional viscous primitive equations of large scale ocean and atmosphere dynamics,~\emph{Ann. Math.},~166~(2007)~245-267.
\bibitem{ct2012}C.~Cao,~E.S.~Titi,~Global well-posedness of the 3D primitive equations with partial vertical turbulence mixing heat diffusion,~\emph{Commun. Math. Phys.},~310~(2012)~537-568.
\bibitem{ct2003}C.~Cao,~E.S.~Titi,~Global well-posedness and finite-dimensional global attractor for a 3-D planetary geostrophic viscous model,~\emph{Commun.~Pure Appl.~Math.},~56~(2003)~198-233.
\bibitem{dy2020}D.~Fang,~B.~Han,~Global well-posedness for the 3D primitive equations in anisotropic framework,~\emph{J. Math. Anal. Appl.},~484~(2020),~123714,~22~pp.
\bibitem{kf2020}K.~Furukawa,~Y.~Giga,~M.~Hieber,~A.~Hussein,~T.~Kashiwabara,~M.~Wrona,~Rigorous justification of the hydrostatic approximation for the primitive equations by scaled Navier-Stokes equations,~\emph{Nonlinearity},~33~(2020)~6502-6516.
\bibitem{ym2020}Y.~Giga,~M.~Gries,~M.~Hieber,~A.~Hussein,~T.~Kashiwabara,~The hydrostatic Stokes semigroup and well-posedness of the primitive equations on spaces of bounded functions,~\emph{J. Funct. Anal.},~279~(2020),~108561,~46 pp.
\bibitem{dh2016}D.~Han-Kwan,~T.~Nguyen,~Ill-posedness of the hydrostatic Euler and singular Vlasov equations,~\emph{Arch. Ration. Mech. Anal.},~221~(2016)~1317-1344.
\bibitem{ah2016}M.~Hieber,~A.~Hussein,~T.~Kashiwabara,~Global strong~$L^p$~well-posedness of the 3D primitive equations with heat and salinity diffusion,~\emph{J. Differ. Equ.},~261~(2016)~6950-6981.
\bibitem{mh2016}M.~Hieber,~T.~Kashiwabara,~Global strong well-posedness of the three dimensional primitive equations in~$L^p$-spaces,~\emph{Arch. Ration. Mech. Anal.},~221~(2016)~1077-1115.
\bibitem{si2021}S.~Ibrahim,~Q.~Lin,~E.S,~Titi,~Finite-time blowup and ill-posedness in Sobolev spaces of the inviscid primitive equations with rotation,~\emph{J. Differ. Equ.},~286~(2021)~557-577.
\bibitem{ju2017}N.~Ju,~On~$H^2$~solutions and~$z$-weak solutions of the 3D primitive equations,~\emph{Indiana Univ. Math. J.},~66~(2017)~973-996.
\bibitem{gm2006}G.M.~Kobelkov,~Existence of a solution in the large for the 3D large-scale ocean dynamics equations,~\emph{C. R. Math. Acad. Sci. Paris},~343~(2006)~283-286.
\bibitem{ik2014}I.~Kukavica,~Y.~Pei,~W.~Rusin,~M.~Ziane,~Primitive equations with continuous initial data,~\emph{Nonlinearity},~27~(2014)~1135-1155.
\bibitem{ik2007}I.~Kukavica,~M.~Ziane,~The regularity of solutions of the primitive equations of the ocean in space dimension three,~\emph{C. R. Math. Acad. Sci. Paris},~345~(2007)~257-260.
\bibitem{mz2007}I.~Kukavica, M.~Ziane,~On the regularity of the primitive equations of the ocean,~\emph{Nonlinearity},~20~(2007)~2739-2753.
\bibitem{jl1992}J.L.~Lions,~R.~Temam,~S.~Wang,~New formulations of the primitive equations of atmosphere and applications,~\emph{Nonlinearity},~5~(1992)~237-288.
\bibitem{rt1992}J.L.~Lions,~R.~Temam,~S.~Wang,~On the equations of the large scale ocean,~\emph{Nonlinearity},~5~(1992)~1007-1053.
\bibitem{sw1995}J.L.~Lions,~R.~Temam,~S.~Wang,~Mathematical theory for the coupled atmosphere-ocean models,~\emph{J. Math. Pures Appl.},~74~(1995)~105-163.
\bibitem{lt2019}J.~Li,~E.S.~Titi,~The primitive equations as the small aspect ratio limit of the Navier-Stokes equations:~rigorous justification of the hydrostatic approximation,~\emph{J. Math. Pures~Appl.},~124~(2019)~30-58.
\bibitem{jl2017}J.~Li,~E.S.~Titi,~Existence and uniqueness of weak solutions to viscous primitive equations for a certain class of discontinuous initial data,~\emph{SIAM J. Math. Anal.},~49~(2017)~1-28.
\bibitem{yu2022}J.~Li,~E.S.~Titi,~G.~Yuan.~The primitive equations approximation of the anisotropic horizontally viscous Navier-Stokes equations,~\emph{J. Differ. Equ.},~306~(2022)~492-524.
\bibitem{li2022}J.~Li,~G.~Yuan.~Global well-posedness of $z$-weak solutions to the primitive equations without vertical diffusivity,~\emph{J. Math. Phys.},~63~(2022),~24 pp.
\bibitem{am2003}A.~Majda,~Introduction to PDEs and Waves for the Atmosphere and Ocean,~American Mathematical Society,~Providence,~RI,~2003.
\bibitem{jp1987}J.~Pedlosky,~Geophysical Fluid Dynamics,~second edition,~Springer,~New York,~1987.
\bibitem{pz2021}X.~Pu,~W.~Zhou,~Rigorous derivation of the full primitive equations by scaled Boussinesq equations,~arXiv:~2105.10621.
\bibitem{mr2009}M.~Renardy,~Ill-posedness of the hydrostatic Euler and Navier-Stokes equations,~\emph{Arch. Ration. Mech. Anal.},~194~(2009)~877-886.
\bibitem{jc2016}J.C.~Robinson,~J.L.~Rodrigo,~W.~Sadowski,~The Three-Dimensional Navier-Stokes Equations:~Classical Theory,~Cambridge University Press,~Cambridge,~2016.
\bibitem{ds1996}D.~Seidov,~An intermediate model for large-scale ocean circulation studies,~\emph{Dynam. Atmos. Oceans},~25~(1996)~25-55.
\bibitem{tt2010}T.~Tachim Medjo,~On the uniqueness of $z$-weak solutions of the three-dimensional primitive equations of the ocean,~\emph{Nonlinear Anal. Real World Appl.},~11~(2010)~1413-1421.
\bibitem{rt1977}R.~Temam,~Navier Stokes Equations:~Theory and Numerical Analysis,~North-Holland Publishing Co.,~Amsterdam~New York~Oxford,~1977.
\bibitem{gk2006}G.K.~Vallis,~Atmospheric and Oceanic Fluid Dynamics,~Cambridge University Press,~Cambridge,~2006.
\bibitem{wm1986}W.M.~Washington,~C.L.~Parkinson,~An Introduction to Three Dimensional Climate Modeling,~Oxford University Press,~Oxford,~1986.
\bibitem{tk2015}T.K.~Wong,~Blowup of solutions of the hydrostatic Euler equations,~\emph{Proc. Amer. Math. Soc.},~143~(2015)~1119-1125.
\end{thebibliography}
\end{document}